\documentclass{amsart}
\usepackage[T1]{fontenc}
\usepackage{amsmath, amsthm, amssymb, mathrsfs, mathtools, xcolor}
\usepackage{enumerate, enumitem, hyperref}
\usepackage{resmes}
\usepackage{esint}
\usepackage{comment}
\usepackage{braket}
\numberwithin{equation}{section}

\theoremstyle{plain}
\newtheorem{theorem}{Theorem}[section]
\newtheorem{lemma}[theorem]{Lemma}
\newtheorem{corollary}[theorem]{Corollary}
\newtheorem{proposition}[theorem]{Proposition}

\newtheorem{maintheorem}{Theorem}

\theoremstyle{definition}
\newtheorem{definition}[theorem]{Definition}
\newtheorem{remark}[theorem]{Remark}

\numberwithin{equation}{section}

\newcommand{\scS}{\mathscr{S}}
\newcommand{\scD}{\mathscr{D}}

\newcommand{\scH}{\mathscr{H}}

\newcommand{\scQ}{\mathscr{Q}}

\newcommand{\scB}{\mathscr{B}}
\newcommand{\scG}{\mathscr{G}}

\newcommand{\scL}{\mathscr{L}}

\newcommand{\scJ}{\mathscr{J}}
\newcommand{\scE}{\mathscr{E}}
\newcommand{\scY}{\mathscr{Y}}

\newcommand{\cB}{\mathcal{B}}
\newcommand{\cN}{\mathcal{N}}
\newcommand{\cD}{\mathcal{D}}

\newcommand{\R}{\mathbb{R}}
\newcommand{\Z}{\mathbb{Z}}
\newcommand{\N}{\mathbb{N}}

\newcommand{\GH}{\mathrm{GH}}
\newcommand{\Co}{\mathop\mathrm{Co}}
\newcommand{\WC}{\mathrm{WC}}
\newcommand{\BL}{\mathrm{BiLip}}
\newcommand{\VBP}{\mathrm{VBP}}
\newcommand{\Lip}{\mathrm{Lip}}
\newcommand{\osc}{\mathrm{Osc}}

\newcommand{\diam}{\mathrm{diam}}
\newcommand{\dist}{\mathrm{dist}}

\newcommand{\spt}{\mathrm{spt}}
\newcommand{\md}{\mathrm{md}}
\newcommand{\td}[1]{\tilde{#1}}
\newcommand{\qn}{\|\cdot\|_Q}

\newcommand{\lebb}{\mathscr{L}^n}
\newcommand{\cd}{\text{cd}}

\AtEndDocument{\bigskip{\footnotesize%
  \textsc{Simons Laufer Mathematical Sciences Institute, Berkeley, CA 94703} \par  
  \textit{E-mail address}: \texttt{jkrandel@ucdavis.edu}\\
  Jared Krandel was partially supported by the National Science Foundation under Grants No. DMS-1763973. This material is also based upon work supported by the National Science Foundation under Grant No. DMS-1928930 while the author was in residence at the Simons Laufer Mathematical Sciences Institute (formerly MSRI) in Berkeley, California, during the Fall 2024 semester.
}}

\title[The WCD and alpha numbers]{Weak Carleson conditions in uniformly rectifiable metric spaces: the WCD and alpha numbers}
\author{Jared Krandel}
\date{}

\begin{document}
\begin{abstract}
We investigate characterizations of uniformly rectifiable (UR) metric spaces by so-called weak Carleson conditions for flatness coefficients which measure the extent to which Hausdorff measure on the metric space differs from Hausdorff measure on a normed space. First, we show that UR metric spaces satisfy David and Semmes's weak constant density condition, a quantitative regularity property which implies \textit{most} balls in the space support a measure with \textit{nearly constant} density in a neighborhood of scales and locations. Second, we introduce a metric space variant of Tolsa's alpha numbers that measure a local normalized $L_1$ mass transport cost between the space's Hausdorff measure and Hausdorff measure on a normed space. We show that a weak Carleson condition for these alpha numbers gives a characterization of metric uniform rectifiability. We derive both results as corollaries of a more general abstract result which gives a tool for transferring weak Carleson conditions to spaces with very big pieces of spaces with a given weak Carleson condition.
\end{abstract}

\maketitle
\tableofcontents

\section{Introduction}
\subsection{Overview}
Rectifiability properties of sets and measures are fundamental topics of interest in geometric measure theory. Although these properties were initially studied for subsets of Euclidean space, in recent years there have been significant developments in rectifiability properties of metric spaces.
\begin{definition}[$n$-rectifiability]\label{def:rect}
    Let $X$ be a metric space. We say that $E\subseteq X$ is $n$-rectifiable if there exist Borel sets $A_i\subseteq \R^n$ and Lipschitz maps $f_i:A_i\rightarrow X$ such that
    \begin{equation*}
        \scH^n\left(E\setminus\bigcup_if_i(A_i)\right) = 0.
    \end{equation*}
\end{definition} 
The class of \textit{uniformly $n$-rectifiable} (UR) metric spaces give a strictly stronger notion of rectifiability which is \textit{quantitative} in nature.
\begin{definition}[uniform $n$-rectifiability]\label{def:ur}
    A metric space $X$ is \textit{uniformly $n$-rectifiable} if $X$ has \textit{Big Pieces of Lipschitz images of $\R^n$} (BPLI) and there exists a constant $C_0 > 0$ such that $X$ is \textit{Ahlfors $(C_0,n)$-regular}. By having BPLI, we mean there exist constants $L, \theta > 0$ such that for all $x\in X$ and $0 < r < \diam(X)$, there exists an $L$-Lipschitz map $f:A_{x,r}\subseteq B(0,r)\subseteq\R^n\rightarrow X$ such that 
    \begin{equation}
        \scH^n(B(x,r)\cap f(A_{x,r})) \geq \theta r^n.
    \end{equation}
    By being Ahlfors $(C_0,n)$-regular, we mean that for all $x\in X$ and $0 < r < \diam(X)$,
    \begin{equation}
        C_0^{-1}r^n \leq \scH^n(B(x,r)) \leq C_0r^n.
    \end{equation}
\end{definition}
This is a stronger form of $n$-rectifiability in which one enforces Ahlfors $n$-regularity and requires a uniform percentage of the measure of each ball to be covered by a single Lipschitz image. In the case when $X$ is a subset of some Euclidean space with the induced metric, this class of sets were introduced and studied in detail by David and Semmes \cite{DS91}, \cite{DS93}. In their work, David and Semmes provide numerous analytic and geometric characterizations of these sets involving boundedness of singular integral operators, quantitative control over numerous coefficients measuring local non-flatness, quantitative approximation by Lipschitz graphs, and more.  

It has long been an open question whether there exist analogous characterizations of UR metric spaces similar to those studied by David and Semmes in Euclidean spaces. Recent work by Bate, Hyde, and Schul has given significant progress, giving metric versions of the well-known BWGL and corona decomposition \cite{BHS23}. Fassler and Violo have also managed to give a version of the so-called strong geometric lemma characterization of UR for one-dimensional sets \cite{FV23}.

In this paper, we give some progress on characterizations of UR metric spaces by showing (1) UR metric spaces satisfy the \textit{weak constant density condition} of David and Semmes, and (2) metric uniform rectifiability is characterized by a \textit{weak Carleson condition} for a metric variant of Tolsa's $\alpha$ number. We derive both results as an application of a more general, abstract tool for transferring weak Carleson conditions from coefficients on spaces in a given class $\scY$ to spaces with \textit{very big pieces} of spaces in $\scY$. In our applications, we apply this result by first showing that \textit{subsets} of bi-Lipschitz images of $\R^n$ have well-controlled densities and alpha numbers. Then, we use the fact that UR metric spaces have very big pieces of bi-Lipschitz images to apply the abstract transference result

\subsection{The weak constant density condition}
One of many interesting characterizations of $n$-rectifiability in Euclidean spaces involves the Hausdorff density, which measures how the $\scH^n$ measure of small balls around a point compare to the $\scH^n$ measure of balls of equal radius in $\R^n$.
 \begin{theorem}\label{t:euc-dens}
        Let $E\subseteq \R^d$ be $\scH^n$ measurable with $\scH^n(E) < \infty$. The set $E$ is $n$-rectifiable if and only if for $\scH^n$-a.e. $x\in E$,
        \begin{equation}\label{e:rect-density}
            \lim_{r\rightarrow0}\frac{\scH^n(E\cap B(x,r))}{(2r)^n} = 1.
        \end{equation}
    \end{theorem}
The backward direction was proven by Besicovitch for $n=1,\ d=2$ \cite{Be28},\cite{Be38},  Marstrand for $n=2,\ d=3$ \cite{Mar61} and Mattila for general $n\leq d$ \cite{Ma75}. Later, Preiss showed that any measure in $\R^d$ whose $n$-dimensional density merely exists and is positive and finite $\scH^n$-a.e. is $n$-rectifiable, generalizing this result significantly \cite{Pr87}. 

The forward direction of Theorem \ref{t:euc-dens} follows more readily from the almost everywhere differentiability of Lipschitz maps, but it was not until Kircheim gave a notion of metric differentiability for maps from $\R^n$ into $X$ that Theorem \ref{t:euc-dens} received the following one-sided metric space analog.
\begin{theorem}[See \cite{Ki94} Theorem 9]\label{t:met-dens}
    Let $E\subseteq X$ be $n$-rectifiable with $\scH^n(E)<\infty$. Equation \eqref{e:rect-density} holds at $\scH^n$-a.e. $x\in E$.
\end{theorem}
It follows from work of Preiss and Ti\v{s}er \cite{PT92} that the converse of Theorem \ref{t:met-dens} holds when $n=1$, but it remains an interesting and difficult open question whether the converse holds for general $n$.

The weak constant density condition (WCD) provides an analog of \eqref{e:rect-density} in the world of uniform rectifiability. David and Semmes introduced the WCD as a way of quantifying \eqref{e:rect-density} by requiring that in almost every \textit{ball}, there exists a measure supported on the set with \textit{nearly constant} density nearby.
\begin{definition}[weak constant density condition, Carleson sets and measures]\label{def:wcd}
    Let $X$ be  Ahlfors $n$-regular, let $C_0,\epsilon_0 > 0$, and define
    \begin{align}\label{e:wcd-good}
        \scG_\cd(\epsilon_0) &= \Set{ (x,r)\in X\times\R^+ | \begin{array}{l} 
            \exists \mu \text{ a $C_0$-regular measure with }\spt\mu = X\\
            \text{such that }\forall y\in B(x,r),\ 0 < t \leq r,\\
            |\mu(B(y,t)) - t^n| \leq \epsilon_0 r^n 
        \end{array}},\\ \label{e:wcd-bad}
        \scB_\cd(\epsilon_0) &= X\times\R^+\setminus \scG_\cd(\epsilon_0).
    \end{align}
    We say that $X$ satisfies the weak constant density condition if for all $\epsilon_0 > 0$, $\scB_\cd(\epsilon_0)$ is a Carleson set. That is, there exists a constant $C > 0$ such that for all $z\in X$ and $0 < r < \diam(X)$,
    \begin{equation*}
        \int_{B(z,r)}\int_0^r\chi_{\scB_\cd(\epsilon_0)}(x,t)d\scH^n(x)\frac{dt}{t} \leq Cr^n.
    \end{equation*}
    If this holds, we say that $\chi_{\scB_\cd(\epsilon_0)}d\scH^n(x)\frac{dt}{t}$ is a \textit{Carleson measure} and say that $\scB_\cd(\epsilon_0)$ is \textit{$C$-Carleson}.
\end{definition}
For related quantitative conditions involving densities, see \cite{CGLT16}, \cite{AH22}, and \cite{TT15}. The work of David, Semmes, and Tolsa combine to prove the following Theorem:
\begin{theorem}\label{t:euc-wcd}
    Let $E\subseteq\R^d$ be Ahlfors $n$-regular. The set $E$ is uniformly $n$-rectifiable if and only if $E$ satisfies the WCD.
\end{theorem}
David and Semmes proved the forward implication in Chapter 6 of \cite{DS91} using a characterization of uniform rectifiability (condition C2 of \cite{DS91}) more closely related to the boundedness of singular integral operators. We will say more about this when we discuss our result.

They proved the reverse implication only in the case $n = 1,2,$ and $d-1$. Their proof uses the fact that if a measure is very close to having constant density in a large neighborhood of scales and locations, then its support is well-approximated by the support of an \textit{$n$-uniform measure}, a measure $\mu$ for which there exists $c>0$ such that $\mu(B(x,r)) = cr^n$ for all $x\in\spt(\mu)$ and $r > 0$. Because uniform measures in Euclidean space are completely classfied for $n=1,2$ (they are all multiples of Hausdorff measure on a plane) and for $n=d-1$ (they are Hausdorff measure on products of planes and light cones \cite{KP87}), David and Semmes are able to show that a WCD set is very close to flat on most balls which are good for the WCD. The absence of a classification for uniform measures in intermediate dimensions prevented a direct adaptation of their arguments. However, Tolsa completed the proof of the reverse direction in Theorem \ref{t:euc-wcd} in \cite{To15} by replacing elements of David and Semmes's argument specific to their examples of uniform measures with general flatness properties of uniform measures derived by Preiss \cite{Pr87} in addition to new arguments using the Riesz transform.

In general, classifying uniform measures is a difficult open problem, but see \cite{Ni22} for an interesting family of examples. For further studies of uniform measures in Euclidean spaces see \cite{Pr87} (and \cite{DeLe08} for a more gentle presentation of Preiss), \cite{KP02}, \cite{Ni17}, and \cite{Ni19}. For research into uniform measures in the Heisenberg group see \cite{CMT20} and \cite{Me22} and for a related result in $\ell_\infty^3$, see \cite{Lo03}.

Just as Theorem \ref{t:euc-wcd} provides a quantitative analog of Theorem \ref{t:euc-dens}, one might expect a quantitative analog of Kircheim's result, Theorem \ref{t:met-dens}, to hold for \textit{uniformly rectifiable metric spaces}, i.e., metric spaces which are Ahlfors $n$-regular and have big pieces of Lipschitz images of subsets of $\R^n$. In this paper, we provide such an analog by proving the following theorem.
\begin{maintheorem}\label{t:gen-wcd}
    Uniformly $n$-rectifiable metric spaces satisfy the WCD.
\end{maintheorem}

We note here that the naive converse of Theorem \ref{t:euc-wcd} is false: There exist Ahlfors regular metric spaces which satisfy the WCD, yet are not uniformly rectifiable. Indeed, the metric space $(X,d) = (\R,|\cdot|_{\text{Euc}}^{1/2})$ is in fact $2$-uniform: $\scH^2(B(x,r)) = 2r^2$ for all $x\in \R$ and $r \geq 0$, hence $X$ satisfies the weak constant density condition, yet $X$ is purely $2$-unrectifiable (notice that the Hausdorff $2$-density is everywhere 1/2 so that this space does not give a counterexample to the potential converse of Theorem \ref{t:met-dens}). Some different examples of this failure even in the case $n=1$ are given by Bate \cite{Ba23}. He proves that every $1$-uniform metric measure space is either $\R$, a particular union of disjoint circles of radius $d$, or a purely unrectifiable ``limit'' of the circle spaces. These last two spaces are examples of $1$-uniform spaces which are not uniformly rectifiable.

Analyzing connectedness plays a special role in the proof because any $1$-uniform connected component must be locally isometric to $\R$, implying any connected $1$-uniform space is itself $\R$. From these examples, it seems reasonable to think that some connectedness and topological conditions are necessary hypotheses for any type of converse to hold. It also follows from work of Schul \cite{Sc07}, \cite{Sc09}, and Fassler and Violo \cite{FV23} (see also \cite{Ha05}) that any Ahlfors $1$-regular connected subset of a metric space is uniformly $1$-rectifiable, although perhaps adding some form of weaker hypothesis could provide an interesting converse to our result in the one-dimensional case using Bate's classification.

\subsection{Alpha numbers}
The \textit{alpha numbers} are a family of coefficients related to optimal mass transport which were introduced by Tolsa who applied them to problems related to $L^2$ boundedness of operators such as the Riesz transform \cite{To09}. Given a closed ball $B\subseteq\R^d$ and any two finite Borel measures $\sigma,\nu$ in $\R^d$, we define
\begin{equation*}
    \dist_B(\sigma,\nu)\vcentcolon= \sup\left\{ \left|\int fd\sigma - \int f d\nu\right|:\spt(f)\subseteq B,\ \Lip(f)\leq 1 \right\}.
\end{equation*}
Using this distance, Tolsa defined $\alpha$ as follows:
\begin{definition}
    Given an Ahlfors $n$-regular measure $\mu$ and a closed ball $B(x,r)\subseteq\R^d$, define
    \begin{equation}\label{e:euc-alpha}
        \alpha_\mu(B)\vcentcolon= \frac{1}{\ell(Q)^{n+1}}\inf_{c \geq 0, L}\dist_B(\mu,c\scH^n\resmes L)
    \end{equation}
    where the infimum is taken over all $n$-planes $L$ in $\R^d$.
\end{definition}
Tolsa was able to show that a strong Carleson condition for these numbers characterizes uniformly $n$-rectifiable measures in Euclidean space.
\begin{theorem}[\cite{To09} Theorem 1.2]\label{t:tolsa-alph}
    Let $\mu$ be an Ahlfors $n$-regular measure in $\R^d$. Then $\mu$ is uniformly $n$-rectifiable if and only if $\alpha_\mu(x,r)d\mu(x)\frac{dr}{r}$ is a Carleson measure.
\end{theorem}
This number is also closely related to optimal mass transport; through Kantorovich duality, $\dist_B(\sigma,\nu)$ defined above is closely related to $W_1(\sigma,\nu)$, the Wasserstein $1$-distance of $\sigma$ and $\nu$. In \cite{To12}, Tolsa generalizes the above notion of alpha number by defining numbers $\alpha_p$ for any $1 \leq p < \infty$ using a localized Wasserstein $p$-distance in place of $\dist_B$. In fact, it is even true that $\alpha_p$ characterizes Euclidean uniform rectifiability for $1 \leq p \leq 2$ (See \cite{To12} Theorem 1.2).

Studying quantitative rectifiability via alpha numbers has also led to important results in qualitative rectifiability \--- various alpha numbers have been used to give characterizations of $n$-rectifiable measures in Euclidean space. See \cite{ATT20} for a characterization of pointwise doubling $n$-rectifiable measures in terms of a version of $\alpha$, and see \cite{Dam20},\cite{Dam21} for a characterization in terms of a version of $\alpha_2$. 

In this paper, we define a metric variant of $\alpha$ in Definition \ref{def:alpha} and show that a weak Carleson condition (See Definition \ref{def:wcc}) for the coefficient (See Definition \ref{def:co} $\alpha$ characterizes uniformly $n$-rectifiable metric spaces
\begin{maintheorem}\label{t:metric-UR-alpha}
    Let $X$ be an Ahlfors $n$-regular metric space. Then $X$ is uniformly $n$-rectifiable if and only if $\alpha_X\in\Co(X)$ has a weak Carleson condition on $X$.
\end{maintheorem}
It is an interesting open question whether a strong Carleson condition for $\alpha$ as in \ref{t:tolsa-alph} also characterizes metric uniform rectifiability.

\subsection{Outline of the paper}
In section \ref{sec:transfer}, we give a broad definition of a coefficient (such as $\alpha$) and prove Theorem \ref{t:wgl-suff-cond}, an abstract result which gives conditions under which one can ``transfer'' a weak Carleson condition under very big pieces. This result is used in both the proofs of Theorems \ref{t:gen-wcd} and \ref{t:metric-UR-alpha}.

In section \ref{sec:wavelets}, we develop more tools used in the proofs of Theorems \ref{t:gen-wcd} and \ref{t:metric-UR-alpha}. The main results are Lemma \ref{l:wavelet-weak-flat-L2-sand} and Corollary \ref{c:oscillation-carleson} which allow us to quantitatively control the variation of means of $L^\infty$ functions on $\R^n$ over a general class of subsets of the domain which we call sandwichable (the motivating example of such a class is $L$-bi-Lipschitz images of a fixed ball). Corollary \ref{c:oscillation-carleson} is similar to results given by David and Semmes (See \cite{DS93} Lemma IV.2.2.14 and Corollary IV.2.2.19, and see Remark \ref{r:cubes-vs-balls} for a discussion of the difference with our result), although our proof proceeds by contradiction, a method which differs significantly from their proofs. These results are applied to control the variation of the Jacobian of bi-Lipschitz mappings $g:\R^n\rightarrow\Sigma$ over normed balls in its domain, controlling the variation of the $\scH^n$ measure of balls in $\Sigma$. 

In section \ref{sec:app}, we prove Theorems \ref{t:gen-wcd} and \ref{t:metric-UR-alpha}. The structure of both proofs are similar. For the proofs of the Carleson conditions in both theorems, we apply our transference theorem \ref{t:wgl-suff-cond} using Bate, Hyde, and Schul's result that UR spaces have very big pieces of bi-Lipschitz images. The main difficulty is showing that suitable versions of our coefficients defined on \textit{subsets} of bi-Lipschitz images have the necessary Carleson conditions (See Proposition \ref{p:wcd-Sigma} and Proposition \ref{p:bilip-alpha-wcc}.) For the WCD, this method of proof differs significantly from the original Euclidean proof of David and Semmes which uses a Carleson condition for integrals of smooth odd functions defined on the ambient Euclidean space to show that UR subsets have a form of quantitative symmetry from which one can deduce near uniformity of Hausdorff measure (See section 6 of \cite{DS91}.).

\section{Acknowledgements}
The author would like to thank Raanan Schul for reading and giving comments on an early draft of this paper. The author would also like to thank David Bate for helpful conversations surrounding metric notions of alpha number and for encouraging him to formulate Theorem \ref{t:wgl-suff-cond}.

\section{Preliminaries}\label{sec:prelim}
Whenever we write $A \lesssim B$, we mean that there exists some constant $C$ independent of $A$ and $B$ such that $A \leq CB$. If we write $ A \lesssim_{a,b,c} B$ for some constants $a,b,c$, then we mean that the implicit constant $C$ mentioned above is allowed to depend on $a,b,c$. We will sometimes write $A\asymp_{a,b,c} B$ to mean that both $A\lesssim_{a,b,c} B$ and $B\lesssim_{a,b,c} A$ hold. We use the notation $f:E\twoheadrightarrow F$ to mean $f$ is a surjective map from $E$ to $F$.

Let $(X,d)$ be a metric space. For any subset $F\subseteq X$, integer $n\geq 0$, and constant $0 < \delta \leq \infty$, we define
\begin{equation*}
    \scH^n_\delta(F) = \inf\Set{\sum \diam(E_i)^n : F\subseteq\bigcup E_i,\ \diam(E_i) < \delta}
\end{equation*}
where $\diam(E) = \sup_{x,y\in E}d(x,y)$. The Hausdorff $n$-measure of $F$ is defined as
\begin{equation*}
    \scH^n(F) = \lim_{\delta\rightarrow 0}\scH^n_\delta(F).
\end{equation*}
Occasionally, we will specify a subset $\Sigma\subseteq X$ and write $\scH_\Sigma^n = \scH^n\resmes{\Sigma}$. For any $\scH^n$ measurable $A\subseteq X$ and measurable $f:A\subseteq X\rightarrow\R$, we define
\begin{equation*}
    \fint_A f = \frac{1}{\scH^n(A)}\int_A f(x)d\scH^n(x).
\end{equation*}
We let $\cD(\R^n)$ denote the family of dyadic cubes in $\R^n$. For $Q\in\cD(\R^n)$, we let $\ell(Q)$ denote the side length of $Q$. If $R\in\cD(\R^n)$ and $k\in\Z$, we define
\begin{align*}
    \cD(R) &= \Set{Q\in\cD(\R^n) | Q\subseteq R},\\
    \cD_k(R) &= \Set{Q\in\cD(R) | \ell(Q) = 2^{-k}\ell(R)}.
\end{align*}
We will actually need to extend the standard system of dyadic cubes.
\begin{definition}[one-third trick lattices]\label{def:ot-trick}
    The following family of dyadic systems were introduced by Okikiolu \cite{Ok92}. For any $e\in\{0,1\}^n$ and cube $Q_0\in\cD(\R^n)$, define the shifted dyadic lattice
    \begin{align*}
        \cD_j^e(Q_0) &= \Set{Q + \frac{\ell(Q)}{3}e | Q\in\cD_j(Q_0)},\\
        \cD^e(Q_0) &= \bigcup_{j\geq 0}\cD_j^e(\R^n)
    \end{align*}
    and set 
    \begin{equation*}
        \td{\cD}(Q_0) = \bigcup_{e\in\{0,1\}^n} \cD^e(Q_0).
    \end{equation*}
    $\td{\cD}(Q_0)$ has the following property: For any $x\in Q_0$ and $j\geq 0$, there exists $Q\in\td{\cD}(Q_0)$ such that $x\in \frac{2}{3}Q$ (See \cite{Le03} Proposition 3.2).
\end{definition}
We will also need a version of ``cubes'' associated to a metric space. David \cite{Da88} introduced this idea first, and it was later generalized by \cite{Ch90} and \cite{HM12}. The following formulation draws most from the latter two.
\begin{theorem}[Christ-David cubes]\label{t:cd-cubes}
Let $X$ be a doubling metric space. Let $X_{k}$ be a nested sequence of maximal $\rho^{k}$-nets for $X$ where $\rho<1/1000$ and let $c_{0}=1/500$. For each $k\in\Z$ there is a collection $\mathscr{D}_{k}$ of ``cubes,'' which are Borel subsets of $X$ such that the following hold.
\begin{enumerate}[label=(\roman*)]
    \item $X=\bigcup_{Q\in \mathscr{D}_{k}}Q$.
    \item If $Q,Q'\in \mathscr{D}=\bigcup \mathscr{D}_{k}$ and $Q\cap Q'\neq\emptyset$, then $Q\subseteq Q'$ or $Q'\subseteq Q$.
    \item For $Q\in \mathscr{D}$, let $k(Q)$ be the unique integer so that $Q\in \mathscr{D}_{k}$ and set $\ell(Q)=5\rho^{k(Q)}$. Then there is $x_{Q}\in X_{k}$ so that
    \begin{equation*}
        B(x_{Q},c_{0}\ell(Q) )\subseteq Q\subseteq B(x_{Q},\ell(Q))
    \end{equation*}
    and
    \[ X_{k}=\{x_{Q}: Q\in \mathscr{D}_{k}\}.\]
    \item If $X$ is Ahlfors $n$-regular, then there exists $C \geq 1$ such that 
    \[\scH^n(\Set{x\in Q | d(x,X\setminus Q)\leq\eta\rho^k})\lesssim\eta^{1/C}\ell(Q)^n\] for all $Q\in\scD$ and $\eta > 0$.
\end{enumerate}
\end{theorem}
In addition, we define
\begin{equation*}
    B_Q = B(x_Q, \ell(Q)).
\end{equation*}
In analogy to the dyadic cube notation, for any $R\in\scD$ and $k\in\Z$ we also write
\begin{align*}
    \scD(R) &= \Set{Q\in\scD | Q\subseteq R},\\
    \scD_k(R) &= \Set{Q\in\scD(R) | \ell(Q) = \rho^{-k}\ell(R)}.
\end{align*}

\section{Transference of weak Carleson conditions under very big pieces} \label{sec:transfer}
In this section, we prove an abstract lemma that allows one to transfer weak Carleson conditions for a variety of beta-type coefficients from very big pieces to their approximated spaces. We begin by discussing what we mean by a coefficient and by weak Carleson conditions.
\begin{definition}[Beta number-type coefficients]\label{def:co}
    Let $X$ be a metric space. We say that $\beta:X\times[0,\diam(X))\rightarrow\R$ is a \textit{coefficient} on $X$ if there exists an increasing, doubling function $f:\R\rightarrow\R$ with $f(1) = 1$ such that for any $x,y\in X$ and $t<r$ satisfying $B(y,t)\subseteq B(x,r)$, we have
    \begin{equation}\label{e:doubling}
        \beta(y,t) \leq f\left(\frac{r}{t}\right)\beta(x,r).
    \end{equation}
    We denote the set of coefficients on $X$ by $\Co(X)$. We often identify the pair $(x,r)$ with the ball $B(x,r)$ and use the convention $\beta(B(x,r)) = \beta(x,r)$. If $\scD(X)$ is a Christ-David lattice on $X$ and $Q\in\scD(X)$, we define
    \begin{equation*}
        \beta(Q) = \beta(B_Q).
    \end{equation*}
\end{definition}
\begin{definition}[weak Carleson conditions]\label{def:wcc}
    Let $X$ be a doubling metric space, let $\scD(X)$ be a Christ-David lattice for $X$, and let $\beta\in\Co(X)$. We say that $\beta\in\Co(X)$ satisfies a \textit{weak Carleson condition} on $X$ if it has the following property: For all $\epsilon > 0$, there exists a constant $C > 0$ such that for any $R\in\scD(X)$,
    \begin{equation}\label{e:WGL-eq}
        \sum_{\substack{Q\subseteq R \\ \beta(Q) > \epsilon}}\ell(Q)^n \leq C \ell(R)^n.
    \end{equation}
    We define
    \begin{equation*}
        \WC(X) = \{\beta \in \Co(X) : \beta \text{ satisfies a weak Carleson condition on $X$}\}.
    \end{equation*}
\end{definition}
\begin{remark}
    Using the ``monotonicity'' property in the definition of a coefficient, it is standard to show that $\beta\in\Co(X)$ satisfying a weak Carleson condition as in Definition \ref{def:wcc} is equivalent to a form of Carleson measure statement as in Definition \ref{def:wcd}. For $\epsilon > 0$, if we define
    \begin{equation*}
        \scB(\epsilon) = \{(x,t)\in X\times \R^+ : \beta(x,t) > \epsilon\}
    \end{equation*}
    then $\beta$ has a weak Carleson condition on $X$ if and only if $\chi_{\scB(\epsilon)}(x,t)d\scH^n(x)\frac{dt}{t}$ is a Carleson measure for every $\epsilon > 0$.
\end{remark}

We want to use \eqref{e:doubling} to show that it actually suffices to control the value of coefficients on tiny balls $\{c_0B_Q\}_{Q\in\scD}$ in a fixed lattice rather than the full balls. The next two lemmas accomplish this.
\begin{lemma}\label{l:cd-multi-res}
    Let $X$ be a doubling metric space with doubling constant $C_d$. There exists $N(C_d)<\infty$ Christ-David systems of cubes $\{\scD_i\}_{i=1}^N$ for $X$ such that the following holds: For any $x\in X,\ 0 < t < \diam(X)$, there exists $i\in\{1,\ldots, N\}$ and $Q\in\scD_i$ with $t\leq \ell(Q)\leq \frac{5}{\rho c_0}t$ such that $B(x,t)\subseteq \frac{c_0}{2}B_Q$.
\end{lemma}
\begin{proof}
    Fix $\rho < \frac{1}{1000}$. For each $k$, let $\td{X}_k$ be a maximal $c_0\rho^k$-net for $X$. We now iteratively construct maximal $\rho^k$-nets $X_k^1,X_k^2,\ldots,X_k^{N},\ldots$ in the following way. Let $X_k^1$ be a completion of a maximal $\rho^k$-separated subset of $\td{X}_k$ to a maximal $\rho^k$-net for $X$. Given $X_k^i$ for any $i > 0$, construct $X_k^{i+1}$ by completing a maximal $\rho^k$-separated subset of $Y_k^i \vcentcolon= \td{X}_k\setminus(X_k^1\cup X_k^{2}\cup \ldots \cup X_k^{i})$ to a maximal $\rho^k$-net for $X$. We claim that this process terminates in $N(C_d)$ steps, giving for each $k\in\Z$ a collection of maximal $\rho^k$-nets $X_k^1,\ldots,X_k^N$. Indeed, let $B$ be a ball of radius $2\rho^k$. By doubling, there exists $N(C_d) <\infty$ such that $\#(B\cap \td{X}_k) \leq N(C_d)$. Suppose that $\frac{1}{2}B\cap Y_k^j \not=\varnothing$ for some $j > 0$. Then, because $X_{k}^{j+1}$ is maximal, there exists some $x\in B\cap Y_k^{j}$ such that $x\in X_k^{j+1}$. Therefore, $\#(B\cap Y_k^{j+1}) < \#(B\cap Y_k^{j})$ whenever $\frac{1}{2}B\cap Y_k^j\not=\varnothing$. This means $\frac{1}{2}B\cap Y_{k}^{N+1} = \varnothing$ for any such $B$, implying $Y_k^{N+1} = \varnothing$ and $\td{X}_k \subseteq \cup_{i=1}^N X_k^i$ as desired.
    
    We now show that the lemma follows from this. Recall that Theorem \ref{t:cd-cubes} takes as input a collection $\{X_k\}_{k\in\Z}$ of maximal $\rho^k$-nets for $X$ and outputs a system of cubes $\scD$ such that every $x^\alpha_k\in X_k$ is the ``center'' of a cube $Q_k^\alpha\in\scD$ with $B_X(x_k^\alpha,c_05\rho^k) = c_0B_{Q_k^\alpha}\subseteq Q_k^\alpha$. We apply Theorem \ref{t:cd-cubes} to the collection $\{X_k^i\}_{k\in\Z}$ for every $1\leq i \leq N$ and receive a Christ-David system $\scD_i$ such that each point $\td{x}_k\in \td{X}_k$ is the center of some $Q\in\scD_i$ for some $i$. So, let $x\in X$, $0 < t < \diam(X)$, and let $k\in\Z$ such that $c_0\rho^{k-1} \leq t < c_0\rho^k$. Because $\td{X}_k$ is a maximal $c_0\rho^k$-net for $X$, there exists $\td{x}_k\in \td{X}_k$ such that $d(x,\td{x}_k) < c_0\rho^{k}$. Because $\td{X}_k\subseteq \cup_{i=1}^NX_k^i$, there then exists $1 \leq i \leq N$ and $Q\in\scD_i$ such that $\td{x}_k = x_{Q}$ so that $x\in B(x_Q,c_0\rho^k) = \frac{1}{5}B(x_Q,c_0\ell(Q)) = \frac{c_0}{5}B_Q$. Similarly, $\frac{\rho c_0}{5}\ell(Q) \leq t < \frac{c_0}{5}\ell(Q)$.
\end{proof}
\begin{lemma}\label{l:small-ball-wgl}
    Let $X$ be doubling, let $\beta\in\Co(X)$, and let $\scD(X)$ be any Christ-David lattice for $X$. Suppose that the collection
    \begin{equation*}
        \scB_0 = \{Q\in\scD(X) : \beta(c_0B_Q) > \epsilon\}
    \end{equation*}
    is Carleson for any $\epsilon > 0$. Then $\beta\in\WC(X)$.
\end{lemma}
\begin{proof}
    Using Lemma \ref{l:cd-multi-res}, we get a collection of $N(C_0,n)$ Christ-David lattices $\{\scD_i\}_{i=1}^N$ such that for any $x\in X$ and $0 < t < \diam(X)$, there exists $1 \leq i \leq N$ such that there is $Q\in\scD_i$ with $B(x,t)\subseteq \frac{c_0}{2}B_Q$ and $t\asymp \ell(Q)$. Therefore, for any $Q\in\scD(X)$, there exists $\td{Q} \in \cup_i\scD_i$ such that $B_Q\subseteq c_0B_{\td{Q}}$ and $\ell(Q)\asymp \ell(\td{Q})$. Define
    \begin{equation*}
        \scB_{\epsilon'} = \{Q\in\scD(X) : \beta(\td{Q}) > \epsilon'\}.
    \end{equation*}
    Notice that if $Q\in\scD\setminus\scB_{\epsilon'}$, then
    \begin{equation*}
        \beta(Q) = \beta(B_Q) \leq f\left(\frac{c_0\ell(\td{Q})}{\ell(Q)}\right) \beta(c_0B_{\td{Q}}) \lesssim \beta(\td{Q}) \leq C\epsilon' < \epsilon
    \end{equation*}
    as long as $\epsilon'$ is small enough. Therefore, it suffices to show that $\scB_{\epsilon'}$ is Carleson. Fix $R\in\scD(X)$ and observe for each $1 \leq i \leq N$, there exists a collection of at most $N'(C_0,n)$ cubes $\scQ_i$ such that for all $Q\in\scD(R)$, there exists $i$ and $Q'\in\scD_i$ such that $\td{Q}\subseteq Q_i$. We also have that any such $Q_i$ has $\ell(Q_i)\asymp_{C_0,n} \ell(R)$. Using the fact that the mapping $Q\mapsto \td{Q}$ is also at most $C(C_0,n)$ to one, we can estimate
    \begin{align*}
        \sum_{\substack{Q\subseteq R \\ Q\in\scB_{\epsilon'}}}\ell(Q)^n \leq \sum_{\substack{Q\subseteq R \\ Q\in\scB_{\epsilon'}}}\ell(\td{Q})^n \leq C\sum_{i=1}^N\sum_{Q\in\scQ_i}\sum_{\substack{Q'\subseteq Q \\ \beta(Q') > \epsilon'}} \ell(Q')^n \lesssim \sum_{i=1}^N\sum_{Q\in\scQ_i} \ell(Q)^n \lesssim \ell(R)^n.
    \end{align*}
\end{proof}
\begin{remark}
    It is not much more work to show that if one instead defines $\beta(Q) = \beta(AB_Q)$ for some $A\geq 1$, then Lemma \ref{l:small-ball-wgl} holds with constants additionally depending on $A$.
\end{remark}
Reducing weak Carleson conditions to controlling $\beta(c_0B_Q)$ rather than $\beta(B_Q)$ is convenient for us because we are interested in approximation of spaces by very big pieces, and the following lemma gives a simple way of finding large families of cubes which have almost all of their measure contained in a given very big piece $Y$. That is, cubes $Q$ such that $\scH^n(Q\setminus Y) \leq \epsilon\scH^n(Q)$. For such a cube, we can immediately control the measure of $c_0B_Q\subseteq Q$ outside of $Y$ in a way we cannot for $B_Q$.

\begin{lemma}\label{l:many-cubes-big-int}
    Let $X$ be a doubling metric space and let $\scD(X)$ be a Christ-David lattice. Let $\epsilon > 0$, let $F\subseteq X \text{ be $\scH^n$ measurable}$, and let $R\in\scD(X)$ be such that $\scH^n(R\setminus F) \leq \epsilon\scH^n(R)$. Define
    \begin{equation}\label{e:R-tilde}
        \tilde{R} = \Set{x\in R | \begin{array}{l}
            \text{For all $Q\in\scD$ such that $x\in Q \subseteq R$, } \\
            \scH^n(Q\cap F) \geq (1-2\epsilon)\scH^n(Q) \end{array} }.
    \end{equation}
    We have $\scH^n(\tilde{R}) \geq \epsilon\scH^n(R)$.
\end{lemma}
\begin{proof}
    This proof is essentially contained in the proof of Lemma IV.2.2.38 in \cite{DS93}, but we need to be precise about the constant $\epsilon$. If $x\in R\setminus\tilde{R}$, then $x$ is contained in some cube $Q$ such that $\scH^n(Q\cap F) \leq (1-2\epsilon)\scH^n(Q)$. Let $\{Q_i\}_i$ be a maximal disjoint family of such cubes so that $R\setminus\tilde{R} = \bigcup_i Q_i$. Then
    \begin{align}\label{e:Rtilde1}
        \scH^n((R\setminus\tilde{R})\cap F) &= \sum_i\scH^n(Q_i\cap F) \leq (1-2\epsilon)\sum_i\scH^n(Q_i)\\ \nonumber
        &\leq (1-2\epsilon)\scH^n(R\setminus \tilde{R}) \leq (1-2\epsilon)\scH^n(R).
    \end{align}
    On the other hand,
    \begin{align}\label{e:Rtilde2}
        \scH^n((R\setminus\tilde{R})\cap F) &= \scH^n((R\cap F) \setminus \tilde{R}) \geq \scH^n(R\cap F) - \scH^n(\tilde{R}) \\\nonumber
        &\geq (1-\epsilon)\scH^n(R) - \scH^n(\tilde{R}).
    \end{align}
    Combining \eqref{e:Rtilde1} and \eqref{e:Rtilde2} and rearranging gives
\begin{equation*}
    \scH^n(\tilde{R}) \geq (1-\epsilon)\scH^n(R) - (1-2\epsilon)\scH^n(R) = \epsilon \scH^n(R).\qedhere
\end{equation*}
\end{proof}
The last ingredient for our transference theorem is the following abstract version of the John-Nirenberg-Stromberg lemma.
\begin{lemma}[\cite{BHS23} Lemma 4.2.8, \cite{DS93} Lemma IV.1.12]\label{l:stability}
    Let $X$ be an Ahlfors $n$-regular metric space and $\scD$ a system of Christ-David cubes for $X$. Let $\alpha:\scD \rightarrow [0,\infty)$ be given and suppose there are $N,\eta > 0$ such that
    \begin{equation}\label{e:stable-ineq}
        \scH^n\left(\Set{ x\in R | \sum_{\substack{Q\subseteq R \\ x\in Q}}\alpha(Q)\leq N}\right) \geq \eta\ell(R)^n
    \end{equation}
    for all $R\in\scD$. Then,
    \begin{equation*}
        \sum_{Q\subseteq R}\alpha(Q)\ell(Q)^n \lesssim_{N,\eta}\ell(R)^n
    \end{equation*}
    for all $R\in \scD$.
\end{lemma}
For proving weak Carleson conditions, we can apply this lemma with $\alpha(Q) = \chi_{\scB}(Q)$ where $\scB$ is a bad collection of cubes which we want to show has a packing condition. If we can show that \eqref{e:stable-ineq} holds for this choice of $\alpha$, then we will conclude
\begin{equation}
    \sum_{\substack{Q\subseteq R \\ Q\in\scB}}\ell(Q)^n = \sum_{Q\subseteq R}\chi_{\scB}(Q)\ell(Q)^n \lesssim_{N,\eta} \ell(R)^n.
\end{equation}

The theorem of this section is a tool for transferring weak Carleson conditions from \textit{very big pieces} to those pieces' approximated spaces.
\begin{definition}[very big pieces]
    Let $Z$ be a metric space, let $X\subseteq Z$ be Ahlfors $n$-regular, and let $\scY$ be a class of Ahlfors $n$-regular subsets of $Z$. We say that $X\in\VBP(\scY)$ if for every $\epsilon > 0$, there exists $C_0 > 0$ such that for every $x\in X$ and $0 < r < \diam(X)$, there exists an Ahlfors $n$-regular space $Y\in\scY$ with constant $C_0$ such that
    \begin{equation*}
        \scH^n(B(x,r)\cap X\cap Y) \geq (1-\epsilon)\scH^n(B(x,r)\cap X).
    \end{equation*}
\end{definition}
\begin{theorem}[transference under VBP]\label{t:wgl-suff-cond}
    Let $X\subseteq Z$ be Ahlfors $n$-regular and let $\mathscr{Y}$ be a class of Ahlfors $n$-regular spaces in $Z$ such that $X\in\VBP(\mathscr{Y})$. Suppose $\beta\in\Co(X)$ is such that for any $Y\in\mathscr{Y}$, there exists a coefficient $\beta_Y\in\Co(Z)$ such that $\beta_Y\in\WC(X)\cup\WC(Y)$ with constant depending only on $n$, the regularity constants for $X$ and $Y$, and the $\VBP$ constant such that for any $Q\in\scD(X)$,
        \begin{equation}\label{e:transfer}
            \beta(c_0B_Q) \lesssim \frac{\scH^n(Q\setminus Y)}{\ell(Q)^n} + \beta_Y(c_0B_Q).
        \end{equation} \label{i:X-cond}
    Then $\beta\in \WC(X)$.
\end{theorem}
\begin{remark}
    If $Q$ is such that $\scH^n(Q\setminus Y) \gtrsim \ell(Q)^n$, then \eqref{e:transfer} is trivially satisfied and there is no need for $\beta_Y$ near $Q$. Although we take $\beta_Y\in\WC(X)\cup\WC(Y)$ for technical reasons, in practice $\beta_Y$ is only used in places where $X$ and $Y$ have some overlap. One should think of $\beta_Y$ as being a controlled coefficient on $Y$ which only sees $Y\cap X$.
\end{remark}
\begin{proof}[Proof of Theorem \ref{t:wgl-suff-cond}]
    Let $\scD(X)$ be a Christ-David lattice for $X$ and fix $0 < \epsilon' < \epsilon$. By Lemma \ref{l:small-ball-wgl}, it suffices to show that the collection
    \begin{equation*}
        \scB_\epsilon = \{Q\in\scD(X) : \beta(c_0B_Q) > \epsilon\}
    \end{equation*}
    is Carleson. By Lemma \ref{l:stability}, it further suffices to show that there exist $N,\eta > 0$ such that for all $R\in\scD(X)$,
    \begin{equation}\label{e:john-niren}
        \scH^n\left(\Set{ x\in R | \sum_{\substack{Q\subseteq R \\ x\in Q}}\chi_{\scB}(Q)\leq N}\right) \geq \eta\ell(R)^n.
    \end{equation}
    Fix $R\in\scD(X)$ and let $Y\in\scY$ be such that
    \begin{equation*}
        \scH^n(B_R\cap X \setminus Y) \leq \epsilon'\scH^n(R).
    \end{equation*}
    Lemma \ref{l:many-cubes-big-int} implies that the set
    \begin{equation*}
        \tilde{R} = \Set{x\in R | \begin{array}{l}
                \text{For all $Q\in\scD$ such that $x\in Q \subseteq R$, } \\
                \scH^n(Q\cap Y) \geq (1-2\epsilon')\scH^n(Q) \end{array} }
    \end{equation*}
    has $\scH^n(\tilde{R}) \geq \epsilon' \scH^n(R)$. Let $\scD(\tilde{R}) = \{Q\in\scD(X) : \exists x\in\tilde{R},\ x\in Q \subseteq R\}$. For any $Q\in\scD(\tilde{R})$, we have that $c_0B_Q\cap X\subseteq Q$, so that 
    \begin{equation*}
        \scH^n(c_0B_Q\cap X\setminus Y) \leq 2\epsilon'\scH^n(Q) \lesssim \epsilon'\ell(Q)^n.
    \end{equation*}
    Assume first that $\beta_Y\in\WC(X)$ satisfies \eqref{e:transfer} and define 
    \begin{align*}
        \tilde{\scG} &= \{Q\in\scD(\tilde{R}): \beta_Y(c_0B_Q) \leq \epsilon'\},\\
        \td{\scB} &= \td{\scD}(R)\setminus \td{\scG}.
    \end{align*}
    Notice that for any $Q\in \tilde{\scG}$, we have
    \begin{equation*}
        \beta(c_0B_Q) \lesssim \frac{\scH^n(c_0B_Q\cap X\setminus Y)}{\ell(Q)^n} + \beta_Y(c_0B_Q) \lesssim \epsilon' + \epsilon' \leq \epsilon
    \end{equation*}
    as long as $\epsilon'$ is small enough, implying that $\chi_{\scB_\epsilon}(Q) \leq \chi_{\tilde{\scB}}(Q)$ for any $Q\in\tilde{\scD}$. Therefore, for any $x\in\td{R}$
    \begin{equation*}
        \sum_{\substack{Q\subseteq R \\ x\in Q}}\chi_{\scB_\epsilon}(Q) \leq \sum_{\substack{Q \subseteq R \\ x\in Q}}\chi_{\td{\scB}}(Q)
    \end{equation*}
    This means that the left-hand side of \eqref{e:john-niren} is bounded below by the same expression with $R$ replaced by $\tilde{R}$ and $\scB$ replaced by $\tilde{\scB}$. We now focus on proving this version of \eqref{e:john-niren}. We can use Chebyshev's inequality to estimate
    \begin{align}\nonumber
        \scH^n\left(\Set{x\in\tilde{R} | \sum_{\substack{Q\subseteq R \\x\in Q}}\chi_{\tilde{\scB}}(Q) > N}\right) &\leq \frac{1}{N}\int_{\tilde{R}}\sum_{\substack{Q\subseteq R \\x\in Q}}\chi_{\tilde{\scB}}(Q) \lesssim \frac{1}{N}\sum_{Q \in \scD(\td{R})\cap\tilde{\scB}}\ell(Q)^n\\
        &\leq \frac{1}{N}\sum_{\substack{Q\subseteq R \\ \beta_Y(c_0B_Q) > \epsilon'} }\ell(Q)^n \lesssim \frac{1}{N}\ell(R)^n \nonumber
    \end{align}
    where the final inequality follows from the fact that $\beta_Y$ satisfies a weak Carleson condition. The result follows by taking $N$ sufficiently large since the left hand side of \eqref{e:john-niren} is bounded below by
    \begin{equation*}
        \scH^n(\td{R}) - \frac{C}{N}\ell(R)^n \geq \epsilon\scH^n(R) - \frac{C}{N}\ell(R)^n \gtrsim_\epsilon \ell(R)^n.
    \end{equation*}
    This completes the proof in the case $\beta_Y\in\WC(X)$. If instead $\beta_Y\in\WC(Y)$, the only additional complication is in showing $\td{\scB}$ is Carleson. This follows because for each $Q\in\scD(X)$ with $c_0B_Q\cap Y \not=\varnothing$, there exists $Q'\in\scD(Y)$ with $\ell(Q') \lesssim \ell(Q)$ and $5B_{Q'}\supseteq B_Q$ where the mapping $Q\mapsto Q'$ is at most $C$ to one with $C$ depending on $n$ and the regularity constant for $Y$. The fact that $\td{\scB}$ is Carleson now follows from the fact that $\beta_Y(c_0B_Q)\lesssim \beta_Y(10B_{Q'})$ and the Carleson condition for $\beta_Y$ on $Y$.
\end{proof}
\section{Oscillation of means}\label{sec:wavelets}
In this section, we introduce the notion of a sandwichable family of sets, review necessary facts about dyadic decompositions of $L^2$ functions, and prove Lemma \ref{c:oscillation-carleson}, one of our main tools for the proof of Theorem \ref{t:gen-wcd}. 

\begin{definition}
    We follow the presentation of \cite{To12}. Given $h:\R^n\rightarrow\R$ and $Q\in\cD(\R^n)$, define
    \begin{equation*}
        \Delta_Qh(x) = \begin{cases}
            \fint_Ph(z)dz - \fint_Qh(z)dz & \text{ if $x\in P$, where $P$ is a child of $Q$},\\
            0 & \text{ otherwise}
        \end{cases}
    \end{equation*}
    If $h\in L^2(\R^n)$, then 
    \begin{equation*}
        h = \sum_{Q\in\cD(\R^n)}\Delta_Qh\quad \text{ and } \quad h\chi_Q = \fint_Q h + \sum_{R\subseteq Q}\Delta_Rh
    \end{equation*}
    where the sums converge in $L^2$ and $\langle \Delta_Qh, \Delta_{Q'}h\rangle_{L^2} = 0$ when $Q\not= Q'$ so that $\|h\|_2 = \sum_{Q\in\cD} \|\Delta_Qh\|_2^2$. One can view $\Delta_Qh$ as a projection of $h$ onto the subspace of $L^2$ formed by the Haar wavelets $h_Q^{\epsilon},\ \epsilon\in\{0,1\}^n\setminus\{(0,0,\ldots,0)\}$ associated to $Q$.
\end{definition}
We now use the wavelet-like decomposition of $h$ to define coefficients $\Delta_k^h(Q)$ which, roughly speaking, measure the variation in means of $h$ from $Q$ through to its $k$-th generation descendants.
\begin{definition}
For any $k\in\N$, $Q\in\cD(\R^n)$, and $h\in L^2(\R^n)$, define
\begin{equation}
    \Delta_k^h(Q)^2 = \sum_{j=0}^k\sum_{R\in\cD_j(Q)}\|\Delta_Rh\|_2^2.
\end{equation}
\end{definition}

\begin{remark}[Properties of $\Delta_k^h$]\label{r:delta-props}
    Notice that if $h\in L^\infty(\R^n)$, then $\Delta^h_k$ has a type of strong Carleson condition since for any $Q_0\in\cD(\R^n)$.
    \begin{equation*}
        \sum_{Q\subseteq Q_0}\Delta^h_k(Q)^2 = \sum_{Q\subseteq Q_0}\sum_{j=0}^k\sum_{R\in \cD_j(Q)}\|\Delta_Rh\|_2^2 \lesssim_{k,n} \sum_{R\subseteq Q_0}\|\Delta_Rh\|_2^2\lesssim_{\|h\|_\infty} \ell(Q_0)^n.
    \end{equation*}
    This implies the weak Carleson condition
    \begin{equation}\label{e:wavelet-carleson}
        \sum_{\substack{Q\subseteq Q_0 \\ \Delta_k^h(Q) >\delta \ell(Q)^{n/2}}}\ell(Q)^n \lesssim_{\delta}\sum_{Q\subseteq Q_0}\Delta_k^h(Q)^2 \lesssim_{k,n,\|h\|_\infty} \ell(Q_0)^n.
    \end{equation}
    $\Delta^h_k$ also scales appropriately in the following manner: Let $Q, \td{Q} \in\cD(\R^n)$ and let $T:\R^n\rightarrow\R^n$ be the affine map sending $\td{Q}$ onto $Q$ by
    \begin{equation}\label{e:affine-map}
        T(x) = x_Q + \left(\frac{x - x_{\td{Q}}}{\ell(\td{Q})}\right)\ell(Q)
    \end{equation}
    where $x_{\td{Q}}$ is the center of $\td{Q}$. Let $\td{h}\in L^2(\td{Q})$ and set $h = \td{h}\circ T^{-1}$. Notice that
    \begin{align*}
        \|h\|_{L^2(Q)}^2 = \int_{T(\td{Q})}(\td{h}\circ T^{-1})^2 = \int_{\td{Q}} \td{h}^2 \frac{\ell(Q)^n}{\ell(\td{Q})^n} = \frac{\ell(Q)^n}{\ell(\td{Q})^n}\|\td{h}\|_2^2.
    \end{align*}
    Similarly, notice that if $V\subseteq Q$ and $\td{V}\subseteq\td{Q}$ with $T(\td{V}) = V$, then
    \begin{align*}
        \|\Delta_Vh\|_2^2 &= \int_V(\Delta_Vh(x))^2dx = \int_{T(\td{V})}(\Delta_{\td{V}}\td{h}(T^{-1}(x)))^2dx\\
        &= \left(\frac{\ell(Q)}{\ell(\td{Q})}\right)^n\int_{\td{V}}(\Delta_{\td{V}}\td{h})^2dx = \left(\frac{\ell(Q)}{\ell(\td{Q})}\right)^n\|\Delta_{\td{V}}\td{h}\|_2^2
    \end{align*}
    which gives $\Delta_k^h(Q)^2 = \left(\frac{\ell(Q)}{\ell(\td{Q})}\right)^n\Delta_k^{\td{h}}(\td{Q})^2$.
\end{remark}

The $\Delta_k^h$ coefficients are useful for controlling the oscillation of averages of $h$ over certain subsets in the domain. The most abstract families of subsets which we control averages over here are called $\triangle$-compact families.
\begin{definition}[$d_\triangle,\ \triangle$-compactness]
    Define a function $d_{\triangle}$ on pairs of Lebesgue measurable subsets of $\R^n$ by 
    \begin{equation*}
        d_{\triangle}(A,B) = \scL(A\triangle B).
    \end{equation*}
    It is well-known that $d_{\triangle}$ is a pseudometric which becomes a metric on equivalence classes of subsets formed by the relation $A\sim B \iff \scL(A\triangle B) = 0$. We say that a family $\scE$ of $\scL$-measurable subsets of $\R^n$ is \textit{$\triangle$-compact} if it is subsequentially compact with respect to $d_\triangle$. That is, for every sequence $E_j\in\scE$, there exists a subsequence $E_{j_k}$ such that there exist a $\scL$-measurable $E\subseteq\R^n$ such that $d_\triangle(E_{j_k},E)\rightarrow 0$ as $k\rightarrow\infty$.
\end{definition}
Roughy speaking, any class of $\triangle$-compact sets with Lebesgue measure bounded from below have a subsequence whose measure concentrates around some limit set. This definition arises in our context  because for any such concentrating sequence one can ``exchange'' integrals over any $E_{j_k}, E_{j_m}$ for large $k$ and $m$ with small cost. It will actually be most convenient to introduce another class of sets that is equivalent to the class of $\triangle$-compact sets (see Lemma \ref{l:symm-sand}) but that more directly lays out the ``concentration'' properties we want for a well-chosen subsequence in the proof of Lemma \ref{l:wavelet-weak-flat-L2-sand}.

\begin{definition}[sandwichable families]
    Let $L \geq 1$ and let $\scE_L$ be a family of $\scL$-measurable subsets of $[0,1]^n$ with $\scL(E) \gtrsim_{L,n} 1$ for all $E\in\scE_L$. 
    We say that the family $\scE_L$ is \textit{sandwichable} if for every sequence $E_j\in\scE_L$, there exists a subsequence $E_{j_k}$ with the following property: There exist families of ``lower'' subsets $\{L_\epsilon\}_{0 < \epsilon < \frac{1}{2}},$ and ``upper'' subsets $\{U_\epsilon\}_{0 < \epsilon < \frac{1}{2}}$ such that:
    \begin{enumerate}
        \item For all $0 < \epsilon_0 < \frac{1}{2}$, there exists $k_0 > 0$ such that for all $k \geq k_0$, we have $L_\epsilon \subseteq E_{j_k}\subseteq  U_\epsilon$, \label{i:outside-msr}
        \item For any $\epsilon' < \epsilon$, $L_\epsilon \subseteq L_{\epsilon'}$ and $U_{\epsilon'}\subseteq U_\epsilon$,\label{i:inc-and-dec}
        \item $\frac{\scL(L_\epsilon)}{\scL(U_\epsilon)} \geq 1-\epsilon$.\label{i:close-msr}
    \end{enumerate}\label{i:sandwich}
    We also call such a sequence \textit{sandwichable}. For any given $Q\in\cD(\R^n)$, let $T_Q$ be the affine map sending $[0,1]^n$ onto $Q$ and define
    \begin{equation*}
        \scE_L(Q) = T_Q(\scE_L).
    \end{equation*}
\end{definition}
The motivating example of a sandwichable or family is that of $L$ bi-Lipschitz images of a fixed $n$-dimensional Euclidean ball (see Lemma \ref{l:bilip-sand}). The results involving sandwichable families that follow are more general than what is necessary for our applications (it would suffice to consider only the family of normed balls $L$-bi-Lipschitz to a fixed Euclidean ball), but we believe these more general results could be of independent interest. We now show this definition is equivalent to the more intuitive $\triangle$-compactness we introduced first.
\begin{lemma}\label{l:symm-sand}
    Let $L > 0$ and let $\scE_L$ be a family of $\scL$-measurable subsets of $[0,1]^n$ such that $\scL(E) \gtrsim_{L,n} 1$ for all $E\in\scE_L$. The family $\scE_L$ is sandwichable if and only if $\scE_L$ is $\triangle$-compact.
\end{lemma}
\begin{proof}
    First, suppose $\scE_L$ is sandwichable and let $E_j\in\scE_L$ be a sandiwchable sequence. Define $E = \bigcap_{0 < \epsilon < 1/2} U_\epsilon$. We want to show that $d_\triangle(E_j,E)\rightarrow 0$. Fix $\epsilon_0 > 0$ and choose $j_0 > 0$ so that property \ref{i:outside-msr} is satisfied for $\epsilon_0$. Then for any $k \geq j_0$
    \begin{align*}
        d_\triangle(E_k,E) &= \scL(E_k\triangle E) = \scL\left(E_k\setminus\bigcap_{0 < \epsilon < 1/2}U_\epsilon\right) + \scL\left(\bigcap_{\epsilon > 0}U_\epsilon\setminus E_k\right)\\
        &\leq 0 + \scL\left(U_{\epsilon_0}\setminus E_k\right) \leq \scL(U_{\epsilon_0}\setminus L_{\epsilon_0})\\
        &\leq \epsilon_0\scL(U_{\epsilon_0}) \leq \epsilon_0.
    \end{align*}
    Now, let $\scE_L$ be sub-sequentially compact with respect to $d_\triangle$. Let $E_j$ be a subsequence with limit $E$ such that $\scL(E_j\triangle E) \leq 2^{-j}$ and fix $\epsilon > 0$. Let $j_0$ be the smallest integer such that $2^{-j_0} \leq c\epsilon$ for $0 < c < 1$ to be determined later. Define
    \begin{align*}
        L_\epsilon &= \bigcap_{j\geq j_0} E_j,\\
        U_\epsilon &= \bigcup_{j \geq j_0}E_j.
    \end{align*}
    It is immediate that $L_\epsilon$ and $U_\epsilon$ satisfy properties \ref{i:outside-msr} and \ref{i:inc-and-dec} in the definition of sandwichable. To verify the final property, notice that
    \begin{align*}
        E\triangle U_\epsilon = \bigcup_{j\geq j_0} E\triangle E_j
    \end{align*}
    and
    \begin{align*}
        E\triangle L_\epsilon &= \left(E\setminus \bigcap_{j\geq j_0} E_j\right) \cup \left(\bigcap_{j\geq j_0} E_j \setminus E\right) \subseteq \left(\bigcup_{j\geq j_0} E\setminus E_j\right) \cup (E_{j_0}\setminus E).
    \end{align*}
    These imply
    \begin{align*}
        \scL(U_\epsilon \setminus L_\epsilon) &\leq d_\triangle(U_\epsilon,L_\epsilon) \leq d_\triangle(U_\epsilon, E) + d_\triangle(E,L_\epsilon) = \scL(E\triangle U_\epsilon) + \scL(E\triangle L_\epsilon) \\
        &\leq \left(\sum_{j\geq j_0} \scL(E\triangle E_j)\right) + \left(\sum_{j \geq j_0}\scL(E\triangle E_j) + \scL(E_{j_0}\triangle E)\right)\\
        &\leq \sum_{j\geq j_0}2^{-j} + 2\sum_{j\geq j_0}2^{-j} \lesssim 2^{-j_0} \leq c\epsilon.
    \end{align*}
    Because $\scL(E) \gtrsim_{L,n} 1$ we have $\scL(L_\epsilon) \gtrsim_{L,n} 1$ for small enough $\epsilon$, and property \ref{i:close-msr} follows upon choosing $c$ small enough.
\end{proof}
We now aim to show, as promised, that bi-Lipschitz images of Euclidean balls are sandwichable.
\begin{definition}
    Let $L \geq 1$ and define
    \begin{equation*}
        \BL_L = \{E\subseteq [0,1]^n : \exists f:B(0,1)\twoheadrightarrow E,\ f \text{ $L$-bi-Lipschitz}\}.
    \end{equation*}
\end{definition}

\begin{lemma}\label{l:bilip-sand}
$\BL_L$ is a sandwichable family.
\end{lemma}
\begin{proof}
    Let $E_j\in\BL_L(Q)$ and $\epsilon > 0$. By Arzela-Ascoli, the corresponding family $f_j:B(0,\ell(Q))\twoheadrightarrow E_j$ of $L$-bi-Lipschitz maps has a convergent subsequence in the sup norm which we relabel as $f_j$. Let $f:B(0,\ell(Q))\twoheadrightarrow E$ be the $L$-bi-Lipschitz limit mapping and define
    \begin{align*}
        L_\epsilon &= \{y\in E : \dist(y,\R^n\setminus E) \geq \epsilon\ell(Q)\},\\
        U_\epsilon &= \{y\in Q : \dist(y, E) \leq \epsilon \ell(Q)\}.
    \end{align*}
    Now, let $k_0$ be such that $\lVert f_k - f \rVert_\infty < \frac{\epsilon}{5}\ell(Q)$ for all $k \geq k_0$. Then for any $y = f_k(x)\in E_k$, we have that
    \begin{equation*}
        \dist(y, E) \leq |f_k(x) - f(x)| \leq \frac{\epsilon}{5}\ell(Q)
    \end{equation*}
    so that $E_k \subseteq U_\epsilon$. Similarly, for any $y = f(x)\in L_\epsilon$ let $y' = f_k(x')\in f_k(\partial B)$ such that $\dist(y,f_k(\partial B)) = |y-y'|$. Then 
    \begin{align*}
        \dist(y,f_k(\partial B)) &= |y - y'| \geq |f(x) - f(x')| - |f(x') - f_k(x')|\\
        &\geq \dist(y, \R^n\setminus E) - \frac{\epsilon}{5}\ell(Q) \geq \frac{4\epsilon}{5}\ell(Q).
    \end{align*}
    Since $f_k(\partial B)$ separates $\R^n$ into the two components $\R^n \setminus E_k$ and $\mathring{E_k}$, the fact that $\dist(y,E_k) \leq |f_k(x) - f(x)| \leq \frac{\epsilon}{5}\ell(Q)$ combined with the above displayed equation implies $y\in E_k$. This shows $L_\epsilon\subseteq E_k$. We now need to estimate $\frac{\scL(L_\epsilon)}{\scL(U_\epsilon)}$. Let $\{x_i\}_{i=1}^N$ be an $\epsilon\ell(Q)$-net for $f(\partial B)$ and observe that $U_\epsilon \setminus L_\epsilon \subseteq N_{\epsilon\ell(Q)}(f(\partial B)) \subseteq \bigcup_{i=1}^N B(x_i,2\epsilon)$. Using the Ahlfors $n-1$-regularity of $f(\partial B)$ and the disjointness of ${B(x_i,\epsilon\ell(Q)/4)}_{i=1}^N$, we get
    \begin{align*}
        \scL(U_\epsilon\setminus L_\epsilon) &\leq \sum_{i=1}^N \scL(B(x_i,2\epsilon\ell(Q))) \lesssim_{L,n} \sum_{i=1}^N\epsilon\ell(Q)\scH^{n-1}(B(x_i,\epsilon\ell(Q)/4)\cap f(\partial B))\\
        &\leq \epsilon\ell(Q) \scH^{n-1}(f(\partial B)) \lesssim_{L,n} \epsilon \ell(Q)^n.
    \end{align*}
    From which we conclude $\frac{\scL(L_\epsilon)}{\scL(U_\epsilon)} \geq 1-C(L,n)\epsilon$ using that $\scL(L_\epsilon)\gtrsim_{L,n}\ell(Q)^n$.
\end{proof}
We will need to control averages over general bi-Lipschitz images of balls as well as over subclasses of normed balls.
\begin{definition}[normed balls]
    Given $L > 0$, we define the set of norms on $\R^n$ which are $L$-bi-Lipschitz to the Euclidean norm by
    \begin{equation*}
        \cN_L = \{\|\cdot\|\ :\ L^{-1}\|x\| \leq |x| \leq L\|x\|\}.
    \end{equation*}
    We define a collection of $L$-bi-Lipschitz normed balls inside $[0,1]^n$ by
    \begin{equation*}
        \cB_L = \{B_{\|\cdot\|}(x,r)\subseteq [0,1]^n\ :\ \|\cdot\|\in\cN_L,\ r \geq L^{-1}\}.
    \end{equation*}
\end{definition}
\begin{remark}
    The fact that balls in $\cB_L$ are uniformly bi-Lipschitz to $B(0,1)$ implies that $\cB_L$ is a sandwichable family by Lemma \ref{l:bilip-sand}.
\end{remark}
The following lemma gives our main tool for controlling averages of $L^2$ functions.
\begin{lemma}\label{l:wavelet-weak-flat-L2-sand}
    Let $L\geq 1$ and let $\scE_L$ be a sandwichable family. For all $\epsilon,M > 0,\ n\in\N$, there exist $k(\epsilon,M,n,\scE_L)\in\N$ and $\delta(\epsilon, M, n, \scE_L) > 0$ such that the following holds: 
    Suppose $h\in L^2(\R^n)$ and $Q\in\cD(\R^n)$ are such that
    \begin{enumerate}
        \item $h \geq 0$ $\scL^n$-a.e.,
        \item $\|h\|_{L^2(Q)}^2 \leq M \ell(Q)^n$,
        \item $\Delta_k^h(Q)^2 \leq \delta \ell(Q)^n$.
    \end{enumerate}
    Then, for any $E\in\scE_L(Q)$, we have 
    \begin{equation}\label{e:close-averages-L2}
        \left| \fint_{E} h - \fint_Q h \right| \leq \epsilon.
    \end{equation}
\end{lemma}
\begin{proof}
    Suppose the conclusion of the lemma is false. Then there exist $\epsilon, M, L, n$ and a sequence of maps $\td{h}_j\in L^2(\R^n)$ and cubes $\td{Q}_j\in\cD(\R^n)$ with $\td{h}_j \geq 0$, $\|\td{h}_j\|_2 \leq M\ell(Q_j)^n$, and subsets $\td{E}_j\in\scE_L(\td{Q}_j)$ so that \eqref{e:close-averages-L2} does not hold for $\td{h}_j, \td{Q}_j, \td{E}_j$, yet $\Delta_j^{\td{h}_j}(\td{Q}_j) \leq \frac{1}{j}\ell(\td{Q}_j)^{n/2}$. For any $j\in\N$, let $T_j$ be the affine transformation sending $\td{Q}_j$ onto $Q = [0,1]^n$ as in \eqref{e:affine-map} and define $h_j:Q\rightarrow \R$ as in Remark \ref{r:delta-props} by
    \begin{equation*}
        h_j = \td{h}_j\circ T_j^{-1}.
    \end{equation*}
    It follows from Remark \ref{r:delta-props} that $\|h_j\|_{L^2(Q)}^2 \leq M$ and $\Delta_j^{h_j}(Q)^2 \leq \frac{1}{j}$. We also define $E_j=T_j(\td{E}_j) \in \scE_L(Q)$ to be the appropriately translated and scaled copy of $\td{E}_j$. 
    
    By the weak compactness of bounded closed balls in $L^2$, there exists some $h\in L^2(Q)$ such that $h_j\rightharpoonup h$ in $L^2$ for some subsequence of $h_j$. By further refining subsequences we can further assume the subsequence $E_j\in\scE_L$ is sandwichable. Let $c = \int_Q h$ and let $c_j = \int_{Q}h_j$. 

    We will first show that $h = c$ by showing that $\Delta_Vh = 0$ for all $V\subseteq Q$. By weak convergence we have
    \begin{equation}\label{e:means-converge}
        c_j = \int_Q h_j \rightarrow \int_Q h = c.
    \end{equation}
    Write $h_j = c_j + \sum_{R\subseteq Q}\Delta_Rh_j$ and $h = c + \sum_{R\subseteq Q}\Delta_Rh$. Fix $V\subseteq Q$ and observe  
    \begin{align*}
        \int_Q h_j\Delta_Vh &= \int_Q \left(c_j + \sum_{R\subseteq Q}\Delta_Rh_j\right)\Delta_Vh = c_j\int_Q\Delta_Vh + \sum_{R\subseteq Q}\langle \Delta_Rh_j, \Delta_Vh\rangle \\
        &= \langle\Delta_Vh_j,\Delta_Vh\rangle.
    \end{align*}
    where the final equality follows since $\langle \Delta_R f_1, \Delta_V f_2\rangle = 0$ whenever $f_1,f_2\in L^2$ and $R\not= V$. Similarly, we have
    \begin{align*}
        \int_Q h\Delta_Vh = \langle h, \Delta_Vh\rangle = \langle\Delta_Vh, \Delta_Vh\rangle = \|\Delta_Vh\|_2^2
    \end{align*}    
    Using weak convergence again, we get
    \begin{align*}
         \langle\Delta_Vh_j,\Delta_Vh\rangle = \int_Qh_j\Delta_Vh \rightarrow \int_Q h\Delta_Vh = \|\Delta_Vh\|_2^2.
    \end{align*}
    Using Cauchy-Schwarz, we can now conclude that $\|\Delta_Vh\|_2 \leq \lim_j\|\Delta_Vh_j\|_2$. We claim that $\|\Delta_Vh\|_2 = 0$. Indeed, if $j$ is sufficiently large, then both $V\in\cD_{j'}(Q)$ for some $j'\leq j$ and $\|\Delta_Vh\|_2 \leq 2\|\Delta_Vh_j\|_2$. This means
    \begin{equation*}
        \|\Delta_Vh\|_2^2 \leq 4\|\Delta_Vh_j\|^2 \leq 4\sum_{k=0}^j\sum_{R\in\cD_k(Q)}\|\Delta_Rh_j\|_2^2 = 4\Delta_j^{h_j}(Q) \leq \frac{4}{j}.
    \end{equation*}
    for all large $j$. This shows that $\Delta_Vh = 0$ for all $V\subseteq Q$, hence $h = c$ as desired. 

    We will now show how this leads to a contradiction. Let $\eta > 0$ be small, let the families $\{L_{\epsilon'}\}_{\epsilon'}, \{U_{\epsilon'}\}_{\epsilon'}$ satisfy the sandwichable property for the sequence $E_j$, and let $j_0 > 0$ be such that $L_{\eta}\subseteq E_j \subseteq U_{\eta}$ for all $j \geq j_0$. Using the fact that $h_{j}\geq 0$, for all $j \geq j_0$
    \begin{equation}\label{e:h-sandwich}
        \int_{L_\eta}h_{j} \leq \int_{E_{j}} h_{j} \leq \int_{U_\eta}h_{j}
    \end{equation}
    so that
    \begin{equation*}
        \frac{\scL(L_\eta)}{\scL(E_j)}\fint_{L_\eta}h_j - \fint_Q h_j \leq \fint_{E_j}h_j - \fint_Qh_j \leq \frac{\scL(U_\eta)}{\scL(E_j)}\fint_{U_\eta}h_j - \fint_Q h_j.
    \end{equation*}
    Because $\scL(L_\eta) \leq \scL(E_j) \leq \scL(U_\eta)$ and $\frac{\scL(U_\eta)}{\scL(L_\eta)} \leq (1+2\eta)$ for small enough $\eta$, we can assume without loss of generality that
    \begin{equation*}
        \limsup_j\left|\fint_{E_j}h_j - \fint_Qh_j\right| \leq \limsup_j\left|\frac{\scL(U_\eta)}{\scL(L_\eta)}\fint_{U_\eta}h_j - \fint_Q h_j\right| \leq \left|(1+c'\eta)c - c\right| \lesssim_{n} \eta c.
    \end{equation*}
    Since this holds for all $\eta > 0$, we get $\limsup_j\left|\fint_{E_j}h_j - \fint_Qh_j\right| = 0$. On the other hand, by hypothesis
    \begin{equation*}
        \left| \fint_{E_j} h_j - \fint_Qh_j \right| > \epsilon
    \end{equation*}
    for all $j$. This gives a contradiction.
\end{proof}

\begin{corollary}[cf. \cite{DS93} Corollary IV.2.2.19] \label{c:oscillation-carleson}
    Let $L,\epsilon, M > 0,\ \scE_L$ be a sandwichable family, and let $h\in L^\infty(\R^n)$ with $\|h\|_\infty \leq M$. Let
    \begin{equation*}
        \scG = \Set{Q\in\cD(\R^n) | \left| \fint_{E} h - \fint_Q h \right| \leq \epsilon \text{ for all $E\in\scE_L(Q)$}}.
    \end{equation*}
    Then $\scB= \cD(\R^n)\setminus \scG$ is $C(M,n,\epsilon,\scE_L)$-Carleson.
\end{corollary}
\begin{proof}
    Let $h = h^+ - h^-$ be the decomposition of $h$ into positive and negative parts. Choose $k,\delta > 0$ such that the conclusion of Lemma \ref{l:wavelet-weak-flat-L2-sand} holds with constants $\epsilon/2,L,M$. Let 
    \begin{align*}
        \td{\scB} &= \Set{Q\in\cD(\R^n) | \exists E\in\scE_L(Q),\ \left| \fint_E h^+ - \fint_Q h^+ \right| > \frac{\epsilon}{2} \text{ or } \left| \fint_E h^- - \fint_Q h^- \right| > \frac{\epsilon}{2}}.\\
        \td{\scG} &= \cD(\R^n)\setminus \scB
    \end{align*}
    and fix $R\in\cD$. By Lemma \ref{l:wavelet-weak-flat-L2-sand}, $Q\in\scB$ implies either $\Delta_k^{h^+}(Q)^2 > \delta\ell(Q)^n$ or $\Delta_k^{h^-}(Q)^2 > \delta\ell(Q)^n$ so that by \eqref{e:wavelet-carleson}
    \begin{equation*}
        \sum_{\substack{Q\subseteq R \\ Q\in\td{\scB}}} \ell(Q)^n \leq \sum_{\substack{Q\subseteq R \\ \Delta_k^{h^+}(Q)^2 > \delta\ell(Q)^{n}}} \ell(Q)^n + \sum_{\substack{Q\subseteq R \\ \Delta_k^{h^-}(Q)^2 > \delta\ell(Q)^{n}}} \ell(Q)^n \lesssim_{\delta,k,n,M,\scE_L} \ell(R)^n.
    \end{equation*}
    This shows $\td{\scB}$ is Carleson. It now suffices to show that $\td{\scG}\subseteq\scG$ because then we would have $\scB\subseteq \td{\scB}$ which would show that $\scB$ is Carleson. For any $Q\in\td{\scG}$ and $E\in\scE_L(Q)$, we have
    \begin{align*}
        \left| \fint_E h - \fint_Q h\right| \leq \left| \fint_E h^+ - \fint_Q h^+\right| + \left| \fint_E h^- - \fint_Q h^-\right| \leq \frac{\epsilon}{2} + \frac{\epsilon}{2} = \epsilon.
    \end{align*}
    This shows that $Q\in\scG$.
\end{proof}

\begin{remark}\label{r:cubes-vs-balls}
    Suppose that we only want to conclude \eqref{e:close-averages-L2} with normed balls $B\in\cB_L(Q)$ replaced by $Q'\in\cD_j(Q)$ for $j \leq k\in\N$. The following stronger condition holds even without the positivity assumption for $h$: Let $\alpha:\cD(\R^n)\rightarrow \cD(\R^n)$ where $\alpha(Q) \in \cD_j(Q)$. There exists $C_0(k,n) > 0$ such that 
    \begin{equation}
        \sum_{Q\subseteq R}\left|\fint_{\alpha(Q)}h - \fint_Qh\right|^2\ell(Q)^n \leq C_0\|h\|_2^2.
    \end{equation}
    The proof is straightforward: because $\alpha(Q)\in\cD_j(Q)$, there is a chain of at most $k+1$ cubes $\alpha(Q)=Q_j \subseteq Q_{j-1} \subseteq \ldots \subseteq Q_0 = Q$ such that $Q_{j+1}$ is a child of $Q_j$. Therefore, we can use the triangle inequality to write
    \begin{align*}
        \left|\fint_{\alpha(Q)}h - \fint_Qh\right|^2\ell(Q)^n \lesssim_{k,n}\sum_{i=1}^j\left|\fint_{Q_{i}}h - \fint_{Q_{i-1}}h\right|^2\ell(Q_{i-1})^n\leq \sum_{i=1}^j\|\Delta_{Q_i}h\|_2^2. 
    \end{align*}
    Because each cube $Q'\subseteq R$ can appear in at most $N(n,k) < \infty$ chains of the above type, this gives
    \begin{equation*}
        \left|\fint_{\alpha(Q)}h - \fint_Qh\right|^2\ell(Q)^n \lesssim_{n,k}\sum_{Q\subseteq R}\sum_{i=1}^{j(Q)}\|\Delta_{Q_i}h\|_2^2 \lesssim_{n,k}\sum_{Q\subseteq R}\|\Delta_Rh\|_2^2 = \|h\|_2^2.
    \end{equation*}
    The reader should also see \cite{DS93} Lemma IV.2.2.14 for a version of this statement where $\alpha(Q)$ is only required to be ``$N$-close'' to $Q$ rather than contained in $Q$. The main difference in Lemma \ref{l:wavelet-weak-flat-L2-sand} comes from averaging over members of a general sandwichable family rather than just $N$-close dyadic cubes.
\end{remark}

\section{Applications: The WCD and alpha numbers}\label{sec:app}
In this section, we give two applications of Theorem \ref{t:wgl-suff-cond} and the results of Section \ref{sec:wavelets}. First, we prove that uniformly rectifiable metric spaces satisfy the WCD. Then, we define a metric variant of Tolsa's alpha numbers and show that metric uniform rectifiability is characterized by a weak Carleson condition on $\alpha$.

\subsection{Preliminaries: rectifiability in metric spaces}
For both of our applications, we will need the following concepts and results from rectifiability theory in metric spaces.
\begin{definition}[metric derivatives, jacobians]
        Let $f:\R^n\rightarrow \Sigma$ be $L$-Lipschitz. We say a seminorm on $\R^n$ $|Df|(x)$ is a \textit{metric derivative} of $f$ at $x$ if
    \begin{equation*}
        \lim_{y,z\rightarrow x}\frac{d(f(y),f(z)) - |Df|(x)(y-z)|}{|y-x| + |z-x|} = 0.
    \end{equation*}
    Given a seminorm $s$ on $\R^n$, define $\scJ(s)$, the \textit{jacobian of $s$}, by \[\scJ(s) = \alpha(n)n\left(\int_{{\mathbb{S}}^{n-1}}(s(x))^{-n}d\scH^{n-1}(x)\right)^{-1}.\]
\end{definition}
Kircheim used these ideas to prove the following metric analogs of Rademacher's theorem and the area formula for Lipschitz maps from $\R^n$ into metric spaces.
\begin{theorem}[cf. \cite{Ki94} Theorem 2, Corollary 8]
    Let $f:\R^n\rightarrow \Sigma$ be $L$-Lipschitz and let $\scJ_f(x) = \scJ(|Df|(x))$. A metric derivative for $f$ exists at $\lebb$ almost every $x\in\R^n$. In addition, for any Lebesgue integrable function $g:\R^n\rightarrow\R$, 
    \begin{equation*}
        \int_{\R^n}g(x)\scJ_f(x)d\lebb(x) = \int_\Sigma \left(\sum_{x\in f^{-1}(y)}g(x)\right)d\scH^n(y).
    \end{equation*}
\end{theorem}
Azzam and Schul developed the following quantitative measure of how far a function $f$ is from being given by a seminorm.
\begin{definition}
    Let $f:\R^n\rightarrow X$ and $Q\in\cD(\R^n)$. Define
    \begin{equation*}
        \md_f(Q) = \frac{1}{\ell(Q)}\inf_{\|\cdot\|}\sup_{x,y\in Q}\bigg| d(f(x),f(y)) - \|x-y\| \bigg|
    \end{equation*}
\end{definition}
Norms which are close to the infimum in the definition of $\md_f(Q)$ can be thought of as ``coarse'' metric derivatives for $f$ inside $Q$ (note that they are biased towards approximating points whose distances are on the scale of $Q$).
Azzam and Schul proved the following metric quantitative differentiation theorem involving these coefficients.
\begin{theorem}[\cite{AS14} Theorem 1.1]\label{t:qdiff}
    Let $f:\R^n \rightarrow X$ be an $L$-Lipschitz function. Let $\delta > 0$. Then for each $R\in\cD(\R^n)$,
    \begin{equation*}
        \sum_{\substack{Q\in\cD(R)\\ \md_f(3Q) > \delta L}}\ell(Q)^n \leq C(\delta, n)\ell(R)^n.
    \end{equation*}
\end{theorem}
In our language, this theorem says that $\md_f$ satisfies a weak Carleson condition. 

We will also need some of the results from Bate, Hyde, and Schul's foundational paper on uniformly rectifiable metric spaces. One of their primary purposes was to prove an equivalence between uniform rectifiability and a metric version of David and Semmes's \textit{bilateral weak geometric lemma}.
\begin{definition}[Bilateral weak geometric lemma]
    Let $(X,d)$ be a doubling metric space, $x\in X$ and $0<r<\diam(X)$. Let $\|\cdot\|$ be a norm on $\R^n$ and let $\Phi(x,r,\|\cdot\|)$ be the set of Borel maps $\phi:B(x,r)\rightarrow B_{\|\cdot\|}(0,r)$. Define
    \begin{align*}
        \zeta_X(x,r,\phi,\|\cdot\|) &= \frac{1}{r}\sup_{y,z\in B(x,r)}\big|d(y,z) - \|\phi(y)-\phi(z)\|\big|,\\
        \eta_X(x,r,\phi,\|\cdot\|) &= \frac{1}{r}\sup_{u\in B_{\|\cdot\|(0,r)}}\dist_{\|\cdot\|}(u,\phi(B(x,r)),\\
        \xi_X(x,r,\phi,\|\cdot\|) &= \zeta_X(x,r,\phi,\|\cdot\|) + \eta_X(x,r,\phi,\|\cdot\|),\\
        \xi_X(x,r) &= \inf_{\substack{\|\cdot\| \\ \phi\in\Phi(x,r,\|\cdot\|)}}\xi_X(x,r,\phi,\|\cdot\|).
    \end{align*}
    For any $Q\in\scD(X)$, additionally define
    \begin{equation*}
        \xi(Q) = \xi(B_Q).
    \end{equation*}
    We say that a metric space $X$ satisfies the \textit{bilateral weak geometric lemma} (BWGL) if $\xi\in\WC(X)$.
\end{definition}

The following is one of their primary results.
\begin{theorem}\label{t:BHS-BWGL}
    Let $X$ be an Ahlfors $n$-regular metric space. Then $X$ is uniformly $n$-rectifiable if and only if $X$ satisfies the BWGL.
\end{theorem}
We will also need the following result which implies that uniformly rectifiable metric spaces have very big pieces of bi-Lipschitz images.
\begin{theorem}[cf. \cite{BHS23} Theorem B, Proposition 9.0.2]\label{t:vbpbi}
    Let $\epsilon > 0$ and let $X$ be uniformly $n$-rectifiable. There is an $L \geq 1$ depending only on $\epsilon, n$, the Ahlfors regularity constant $C_0$ for $X$, and the BPLI constants for $X$ such that for each $x\in X$ and $r > 0$ there exists $F\subseteq B(x,r)$, satisfying $\scH^n_X(B(x,r)\setminus F) \leq \epsilon r^d$ and an $L$-bi-Lipschitz map $g:F\rightarrow\R^n$.
\end{theorem}
\begin{remark}
    If we embed $X$ isometrically into $\ell_\infty$, then we can take the map $g^{-1}:g(F)\rightarrow F\subseteq\ell_\infty$ above and extend it to an $L'(L,n)$-bi-Lipschitz map $\tilde{g}:\R^n\rightarrow\ell_\infty$ satisfying the same conclusions with respect to the isometric embedding of $X$. (See \cite{BHS23} Lemma 4.3.2 for a proof.)  Define
    \begin{equation*}
        \scS = \{g(\R^n) : g:\R^n\rightarrow\ell_\infty,\exists L \geq 1,\  g\text{ is $L$-bi-Lipschitz}\}.
    \end{equation*}
    It follows that $X\in\VBP(\scS)$. 
\end{remark}

\subsection{The WCD in uniformly rectifiable metric spaces}
We first give a reformulation of the WCD which fits better with the more abstract setting put forth in Section \ref{sec:transfer}.
\subsubsection{Preliminaries with $\osc$}
\begin{definition}
    Let $X$ be a doubling metric space, let $x\in X$, and let $0 < r \leq \diam(X)$. Define
    \begin{equation*}
        \osc(x,r) = \inf_{c \geq 0} \sup_{\substack{y\in B(x,r) \\ 0 \leq t \leq r}}\frac{|\scH^n(B(y,t)) - ct^n|}{r^n}.
    \end{equation*}
    For any $\scH^n$-measurable subset $E\subseteq X$, we additionally define
    \begin{equation*}
        \osc_E(x,r) = \inf_{c \geq 0} \sup_{\substack{y\in B(x,r) \\ 0 \leq t \leq r}}\frac{|\scH^n(B(y,t)\cap E) - ct^n|}{r^n}.
    \end{equation*}
    Notice that $\osc_E(x,r) = 0$ if $\scH^n(B(y,r)\cap E) = 0$.
\end{definition}
\begin{remark}\label{rem:wcd-osc-wgl}
    Let $\scG_{\cd}(\epsilon)$ be as in Definition \ref{def:wcd} and observe that
    \begin{equation*}
        \osc(x,r) \leq \epsilon \implies (x,r)\in \scG_{\cd}(\epsilon)
    \end{equation*}
    Once we know that $\osc\in\Co(X)$, it will follow that an Ahlfors regular space $X$ satisfies the WCD if $\osc$ has a weak Carleson condition on $X$.
\end{remark}
\begin{lemma}\label{l:osc-co}
    For any $\scH^n$-measurable subset $E\subseteq X,\ \osc_E\in\Co(X)$.
\end{lemma}
\begin{proof}
    Fix $x,y\in X$ and $0 < t \leq r \leq \diam(X)$ such that $B(y,t) \subseteq B(x,r)$. Then
    \begin{align*}
        \osc_E(y,t) &= \inf_{c \geq 0} \sup_{\substack{z\in B(z,s) \\ 0 \leq s \leq t}}\frac{|\scH^n(B(y,t)\cap E) - cs^n|}{t^n} \leq \frac{r^n}{t^n}\inf_{c \geq 0} \sup_{\substack{z\in B(x,r) \\ 0 \leq s \leq r}}\frac{|\scH^n(B(z,s)\cap E) - cs^n|}{r^n}\\
        &= \frac{r^n}{t^n}\osc_E(x,r). \qedhere
    \end{align*}
\end{proof}
\begin{remark}\label{rem:osc-co-z}
    If $Z$ is a metric space such that $X\subseteq Z$, then the above proof also shows that $\osc_E\in\Co(Z)$.
\end{remark}
We will use the following lemma to show that we can apply Theorem \ref{t:wgl-suff-cond} to $\osc$ to ``transfer'' the WCD from bi-Lipschitz images to uniformly rectifiable spaces.
\begin{lemma}\label{l:osc-transfer}
    For any $\scH^n$-measurable subset $E\subseteq X$, $x\in X$, and $0 < r < \diam(X)$, we have
    \begin{equation*}
        \osc(x,r) \leq \frac{\scH^n(B(x,2r)\setminus E)}{r^n} + \osc_E(x,r).
    \end{equation*}
\end{lemma}
\begin{proof}
We have
    \begin{align*}
        \osc_X(x,r) &= \inf_{c \geq 0} \sup_{\substack{y\in B(x,r) \\ 0 \leq t \leq r}}\frac{|\scH^n(B(y,t)\cap E) + \scH^n(B(y,t)\setminus E) - ct^n|}{r^n}\\
        &\leq \frac{\scH^n(B(x,2r)\setminus E)}{r^n} + \inf_{c \geq 0} \sup_{\substack{y\in B(x,r) \\ 0 \leq t \leq r}}\frac{|\scH^n(B(y,t)\cap E) - ct^n|}{r^n}\\
        &= \frac{\scH^n(B(x,2r)\setminus E)}{r^n} + \osc_E(x,r)\qedhere.
    \end{align*}
\end{proof}

\subsubsection{$\osc$ on uniformly rectifiable metric spaces}
We begin by studying the WCD on bi-Lipschitz images. We will need to use the good family of dyadic cubes adapted to Christ-David cubes inside a metric bi-Lipschitz image to do analysis in the domain of our bi-Lipschitz map. For the rest of this section, fix an $L$-bi-Lipschitz map $g:\R^n\twoheadrightarrow \Sigma$.
\begin{definition}[$L$-good $I_Q$]\label{def:lgood}
    Let $Q\in\scD(R)$. We call a cube $I_Q\subseteq \R^n$ \textit{$L$-good for $Q$} if the following hold:
    \begin{enumerate}\label{enum:l-good}
        \item[(i)] $\ell(I_Q) \asymp_L \ell(Q)$, \label{i:close-in-size}
        \item[(ii)] $10B_Q \subseteq g(I_Q)$,
    \end{enumerate}
    where the implicit constant in \ref{i:close-in-size}(i) is independent of $Q$ and $I_Q$.
\end{definition}
Now fix a cube $R\in\scD(\Sigma)$. By shifting our initial coordinates on $\R^n$, we can assume that there exists an $L$-good cube $I_R\in\cD(\R^n)$. It is then standard to show that the shifted lattices $\td{\cD} = \td{\cD}(I_R)$ as in Definition \ref{def:ot-trick} contain $L$-good cubes for all $Q\in\scD(R)$.
\begin{lemma}\label{l:dyadic-cubes-approx}
    For each $Q\in\scD(R)$, there exists an $L$-good $I_Q\in\td{\cD}$.
\end{lemma}

We are now ready to set up the proof of Theorem \ref{t:gen-wcd}. We begin by showing that any measurable \textit{subset} of a bi-Lipschitz image satisfies a form of WCD. Indeed, let $E\subseteq \Sigma$ be $\scH^n$-measurable and set $P = g^{-1}(E)$. For any $\delta > 0$, define the following two conditions applicable to any $Q\in\scD(R)$:
\begin{enumerate}
    \item $\md_g(I_Q) \leq \delta$, \label{i:sig-md-small}
    \item For any $B\in\cB_{2L}(I_Q)$, we have 
    \begin{equation*}
        \left|\fint_B \scJ_g\chi_P\ d\scL - \fint_{I_Q}\scJ_g\chi_P\ d\scL\right| < \delta.
    \end{equation*} \label{i:sig-measure-osc}
\end{enumerate}
Now define
\begin{align}\label{e:good-Q-sig}
    \scG(\delta) &= \{Q\in\scD(R) : Q \text{ satisfies \ref{i:sig-md-small} and \ref{i:sig-measure-osc}}\},\\ \label{e:bad-Q-sig}
    \scB(\delta) &= \scD(R) \setminus \scG(\delta).
\end{align}
\begin{lemma}\label{l:carleson-set}
    For any $\delta > 0$, $\scB(\delta)$ is $C(\delta,n,L)$-Carleson.
\end{lemma}
\begin{proof}
    By Lemma \ref{l:dyadic-cubes-approx}, Corollary \ref{c:oscillation-carleson}, and Theorem \ref{t:qdiff} we have
    \begin{align*}
        \sum_{\substack{Q\in\scB(\delta) \\ Q\subseteq R}}\ell(Q)^n &\lesssim \sum_{\substack{Q\subseteq R \\ \text{$Q$ fails \ref{i:sig-measure-osc}}}}\ell(Q)^n + \sum_{\substack{Q\subseteq R \\ \text{$Q$ fails \ref{i:sig-md-small}}}}\ell(Q)^n + \sum_{\substack{Q\subseteq R \\ \ell(Q) > C(L)}}\ell(Q)^n \\
        &\lesssim_{L,n}\sum_{\substack{Q\subseteq R \\ \text{$Q$ fails \ref{i:sig-measure-osc}}}}\ell(I_Q)^n + \sum_{\substack{Q\subseteq R \\ \text{$Q$ fails \ref{i:sig-md-small}}}}\ell(I_Q)^n + C(L,n)\ell(R)^n\\
        &\lesssim_{\delta, L, n} \ell(I_R)^n + \ell(R)^n \lesssim_{L,n} \ell(R)^n. \qedhere
    \end{align*}
\end{proof}
\begin{lemma}\label{l:bilip-osc}
    For any $\epsilon > 0$, there exists $\delta > 0$ such that $Q\in\scG(\delta)$ implies $\osc_E(Q) < \epsilon$.
\end{lemma}
\begin{proof}
    Let $y\in B_Q,\ r \leq \ell(Q)$ and set $y_0 = g^{-1}(y)$. The fact that $\md_g(I_Q) \leq \delta$ implies that there exists a norm $\|\cdot\|_Q$ such that
    \begin{equation}\label{e:close-to-norm}
        \sup_{x,y\in I_Q}\big| d(g(x),g(y)) - \|x-y\|_Q\big| \leq \delta \ell(I_Q).
    \end{equation} 
    We claim that there exists a constant $c_1(n,L)> 0$ such that 
    \begin{equation}\label{e:ball-sandwich}
        B_1 \vcentcolon= B_{\qn}(y_0,(1-c_1\delta)r)\subseteq g^{-1}(B(y,r)) \subseteq B_{\qn}(y_0, (1+c_1\delta)r)=\vcentcolon B_2.
    \end{equation}
    For the first inclusion, let $x_0\in B_1$. By \eqref{e:close-to-norm},
    \begin{equation*}
        d(g(x_0),g(y_0)) \leq \|x_0 - y_0\|_Q + 3\delta\ell(I_Q) \leq (1-c_1\delta)r + C(L,n)\delta r < r
    \end{equation*}
    where the final inequality holds if $c_1$ is large enough. Similarly, let $z_0\in g^{-1}(B(y,r)) \subseteq I_Q$. Then
    \begin{equation*}
        \|z_0 - y_0\|_Q \leq d(f(z_0),f(y_0)) + \delta\ell(3I_Q) \leq r + C(L,n)\delta r \leq (1+c_1\delta)r
    \end{equation*}
    with the same restriction on $\delta$ as above This finishes the proof of \eqref{e:ball-sandwich}. Because $3B_Q \subseteq g(I_Q)$, we immediately have that $B_1,B_2 \subseteq I_Q$ for small enough $\delta$ so that $B_1,B_2\in\cB_{2L}(I_Q)$.

    Let $c_{P,Q} = \fint_{I_Q}\scJ_g\chi_P$ and let $c_{\|\cdot\|_Q} \asymp_{L,n} 1$ be such that $\lebb(B_{\|\cdot\|_Q}(0,r)) = c_{\|\cdot\|_Q}r^n$. We set
    \begin{equation*}
        a_Q = c_{P,Q}c_{\|\cdot\|_Q}
    \end{equation*}
    and plan to show that $a_Q$ is a sufficiently good constant in the definition of $\osc_E(Q)$. By applying the area formula to condition \ref{i:sig-measure-osc}, we get that for any $B\in\cB_L(I_Q)$
    \begin{equation}\label{e:Hn-close-to-leb-normball}
        |\scH^n(g(B\cap P)) - c_{P,Q}\scL(B\cap P)| \leq \delta \scL(B).
    \end{equation}
    
    Applying this inequality to the ball $B_2$ and using \eqref{e:ball-sandwich}, this implies the existence of a constant $c_2(n,L)$ so that
    \begin{align*}
       \scH^n(B(y,r)\cap E) &\leq \scH^n(g(B_2\cap P)) \leq (c_{P,Q}+\delta)\scL(B_2) \leq 
       (c_{P,Q}+\delta)c_{\|\cdot\|_Q}(1+c_1\delta)^nr^n\\ 
       &\leq a_Qr^n + c_2\delta\ell(Q)^n.
    \end{align*}
    A similar computation using $\scH^n(g(B_1\cap P))$ gives a similar lower bound for $\scH^n(B(y,r)\cap E)$. This shows that
    \begin{equation*}
        \osc_E(Q)\ell(Q)^n \leq \left|\scH^n(B(y,r)\cap E) - a_Qr^n\right| \leq c_2(n,L)\delta\ell(Q)^n < \epsilon\ell(Q)^n
    \end{equation*}
    where the final inequality follows by choosing $\delta$ small enough.
\end{proof}

\begin{proposition}\label{p:wcd-Sigma}
    For any $\scH^n$-measurable subset $E\subseteq\Sigma$, $\osc_E\in\WC(\Sigma)$. In particular, $\Sigma$ satisfies the WCD.
\end{proposition}
\begin{proof}
    Let $\epsilon > 0$. By Lemma \ref{l:bilip-osc}, there exist $\delta> 0$ such that $Q\in\scG(\delta)\implies \osc_E(Q) < \epsilon$. Therefore, $\osc_E(Q) \geq \epsilon \implies Q\in\scB(\delta)$. But Lemma \ref{l:carleson-set} shows that $\scB(\delta)$ is Carleson, implying $\osc_E\in\WC(\Sigma)$. When $E=\Sigma$, this is equivalent to the WCD on $\Sigma$ by Remark \ref{rem:wcd-osc-wgl}.
\end{proof}

\begin{theorem}
    Any uniformly $n$-rectifiable metric space $X$ satisfies the WCD.
\end{theorem}
\begin{proof}
We plan to use Theorem \ref{t:wgl-suff-cond}. Without loss of generality, assume $X\subseteq\ell_\infty$ and recall that Theorem \ref{t:vbpbi} implies that $X\in\VBP(\scS)$, so Theorem \ref{t:wgl-suff-cond} implies that it suffices to show that for any $\Sigma\in\scS$ there exists $\beta_\Sigma\in\WC(\Sigma)$ such that $\osc$ satisfies \eqref{e:transfer}. We claim $\osc_{X\cap\Sigma}$ is such a coefficient. Indeed, Lemma \ref{l:osc-co} (with Remark \ref{rem:osc-co-z}) implies that $\osc_{X\cap\Sigma}\in\Co(\ell_\infty)$ and Lemma \ref{l:osc-transfer} implies that for any $Q\in\scD(X)$
\begin{align}
    \osc\left(\frac{c_0}{2}B_Q\right) &\lesssim  \frac{\scH^n(Q\setminus\Sigma)}{\ell(Q)^n} + \osc_{\Sigma\cap X}\left(\frac{c_0}{2}B_Q\right). \label{e:osc-ineq}
\end{align}
By taking $E = \Sigma\cap X$, Proposition \ref{p:wcd-Sigma} implies $\osc_{X\cap\Sigma}\in\WC(\Sigma)$. 
\end{proof}

\subsection{Alpha numbers in Ahlfors $n$-regular metric spaces}
In this section, we define a metric space variant of Tolsa's alpha numbers and prove Theorem \ref{t:metric-UR-alpha}, characterizing uniformly rectifiable metric spaces in terms of a weak Carleson condition for $\alpha$.

\subsubsection{Preliminaries with metric alpha numbers}
\begin{definition}[distances between measures]
    Let $Z$ be a metric space and fix finite Borel measures $\mu,\nu$ in $Z$. Given a ball $B = B(x,r)\subseteq Z$, we define
    \begin{equation}\label{e:Lip-B-def}
        \Lip_1(B) = \{f:Z\rightarrow\R:\Lip(f)\leq 1,\ \forall z\in Z\ |f(z)| \leq \max\{0,r-\dist(z,x)\}\}.
    \end{equation}
    The second condition in \eqref{e:Lip-B-def} is similar to requiring $\spt(f) \subseteq B$, but ensures the correct decay in $|f(z)|$ as $z$ nears the boundary of $Z$. In the case $Z = \R^d$, $\spt(f)\subseteq B$ implies this condition, but this implication does not hold in general because $Z$ could have holes. We also define
    \begin{equation*}
        \dist_B(\mu,\nu) = \sup\left\{ \int fd\mu - \int fd\nu: f\in\Lip_1(B)\right\}.
    \end{equation*}
\end{definition}
\begin{definition}[metric alpha numbers]\label{def:alpha}
    Let $X$ be an Ahlfors $n$-regular metric space and let $x\in X, 0 < r < \diam(X)$. We define $\alpha:X\times[0,\diam(X))\rightarrow\R$ by
    \begin{equation*}
        \alpha_X(x,r) = \frac{1}{r^{n+1}}\inf_{c\geq0,\|\cdot\|,Z}\inf_{\substack{ \iota_1:B(x,r)\rightarrow Z\\ \iota_2:(\R^n,\|\cdot\|)\rightarrow Z}}\dist_{B(\iota_1(x),r)}\left(\iota_1\#[\scH^n_X], c\iota_2\#[\scH^n_{\|\cdot\|}]\right)
    \end{equation*}
    where $\|\cdot\|$ is a norm on $\R^n$, $Z$ ranges over all complete, separable metric spaces, and $\iota_1,\iota_2$ range over all isometric embeddings of $B(x,r)$ and $(\R^n,\|\cdot\|)$ respectively into $Z$. For any $\scH^n$-measurable subset $E\subseteq X$, we also define
    \begin{equation*}
        \alpha_{E,X}(x,r) = \frac{1}{r^{n+1}}\inf_{c\geq0,\|\cdot\|,Z}\inf_{\substack{ \iota_1:B(x,r)\rightarrow Z\\ \iota_2:(\R^n,\|\cdot\|)\rightarrow Z}}\dist_{B(\iota_1(x),r)}\left(\iota_1\#[\scH^n_{E}], c\iota_2\#[\scH^n_{\|\cdot\|}]\right)
    \end{equation*}
    where the infima range over the same parameters as in the definition of $\alpha$. It is a minor, but important technical detail for the proof of Lemma \ref{l:alph-monotone} that we do not take the infimum of $\iota_1$ over the a priori larger family of isometric embeddings $\iota:B(x,r)\cap E\rightarrow Z$.
    
    We note that $\alpha_{E,X}(x,r) = 0$ if $\scH^n(E\cap B(x,r)) = 0$. We will often abuse notation and drop the iotas, identifying subsets of $X$ and $\R^n$ with their isometric images in $Z$. For any Christ-David lattice $\scD(X)$ and $Q\in\scD(X)$, define
    \begin{equation*}
        \alpha_{E,X}(Q) = \alpha_{E,X}(5B_Q).
    \end{equation*}
    For any $Q\in\scD(X)$, let $c_Q,Z_Q,\iota_1^Q,\iota_2^Q,\|\cdot\|_Q$ denote some parameters almost minimizing $\alpha_{E,X}(Q)$ in the sense that
    \begin{equation}\label{e:almost-min}
        \alpha_{E,X}(Q)\ell(Q)^{n+1} \geq \frac{1}{2}\dist_{B(\iota_1^Q(x_Q),\ell(Q))}(\iota_1^Q\#[\scH^n_E], c_Q\iota_2^Q\#[\scH^n_{\|\cdot\|_Q}]).
    \end{equation}
    To shorten formulas, we will often drop the isometric embeddings $\iota_1,\iota_2$. We also define
    \begin{align*}
        \R^n_Q &= \iota_2^Q(\R^n) \subseteq Z_Q,\\
        E_Q &= \iota_1^Q(E\cap 5B_Q) \subseteq Z_Q,\\
        \scH^n_{E_Q} &= \iota_1^Q\#[\scH^n_E],\\
        \scH^n_{\|\cdot\|_Q} &= \iota_2^Q\#[\scH^n_{\|\cdot\|_Q}].
    \end{align*}
\end{definition}
\begin{lemma}\label{l:transfer}
    Let $E\subseteq X$ be $\scH^n$-measurable. For any $x\in X$ and $0 < r < \diam(X)$, we have
    \begin{equation*}
        \alpha_X(x,r) \leq \frac{\scH^n(B(x,r)\setminus E)}{r^n} + \alpha_{E,X}(x,r)
    \end{equation*}
\end{lemma}
\begin{proof}
    Notice that for any fixed $c,\|\cdot\|,Z,\iota_1,\iota_2$, we have
    \begin{align*}
        \dist_{B}(\scH^n_X,c\scH^n_{\|\cdot\|}) &= \sup_{f\in\Lip_1(B(\iota_1(x),r))} \left|\int f d\scH^n_X - c\int f d\scH^n_{\|\cdot\|}\right|\\
        &\leq \sup_{f\in\Lip_1(B(\iota_1(x),r))}\left| \int_{X\cap B \setminus E} f d\scH^n_X\right| + \left|\int_{E\cap B} f d\scH^n_X - c\int f d\scH^n_{\|\cdot\|}\right|\\
        &\lesssim r\scH^n(X\cap B\setminus E) + \dist_{B}(\scH^n_{E},c\scH^n_{\|\cdot\|}).
    \end{align*}
    The result follows by dividing both sides by $r^{n+1}$ and taking infimums.
\end{proof}
\begin{lemma}\label{l:alph-monotone}
    Let $X$ be an Ahlfors $n$-regular metric space, and let $E\subseteq X$ be $\scH^n$-measurable. We have $\alpha_{E,X}\in\Co(X)$.
\end{lemma}
\begin{proof}
    Let $x,y\in X$ and $t<r$ be such that $B(y,t)\subseteq B(x,r)$. Fix any metric space $Z$, norm $\|\cdot\|$, constant $c > 0$, and let $\iota_1:B(x,r)\rightarrow Z$ and $\iota_2:(\R^n,\|\cdot\|)\rightarrow Z$ be any isometric embeddings. For any $f\in\Lip_1(B(\iota_1(y),t))$, we have 
    \begin{align*}
        \frac{1}{t^{n+1}}\int fd\scH^n_E - c\int fd\scH^n_{\|\cdot\|} &\leq \frac{r^{n+1}}{t^{n+1}}\frac{1}{r^{n+1}}\dist_{B(\iota_1(x),r)}(\scH^n_E,c\scH^n_{\|\cdot\|})
    \end{align*}
    where the final line follows because every $f\in \Lip_1(B(\iota_1(y),t))$ can be extended to $\td{f}\in\Lip_1(B(\iota_1(x),r))$ by setting $\td{f} = 0$ in $B(\iota_1(x),r)\setminus B(\iota_1(y),t)$ using the second condition in \eqref{e:Lip-B-def}. By taking infimums, we get $\alpha_{E,X}(y,t) \leq \left(\frac{r}{t}\right)^{n+1}\alpha_{E,X}(x,r)$.
\end{proof}
\begin{remark}\label{rem:alpha-co-z}
    If $Z$ is a metric space such that $X\subseteq Z$, one can naturally view $\alpha_{E,X}$ as a function from $Z\times\R^+$ to $\R$ defined on $(z,t)$ by replacing $E$ above with $B(z,t)\cap E$ and measuring $\dist$ inside $B(\iota_1(x),2r)$ for any $x\in E\cap B(z,t)$. A similar proof then shows $\alpha_{E,X}\in\Co(Z)$. 
\end{remark}
\begin{lemma}[cf. \cite{To09} Lemma 3.1]\label{l:alpha-props}
    Let $X$ be an Ahlfors $n$-regular metric space, and let $E\subseteq X$ be $\scH^n$-measurable. There exists $c_1 > 0$ such that for any $Q\in\scD(X)$, if $\scH^n(E\cap c_0B_Q) \gtrsim \ell(Q)^n$ and $\alpha_{E,X}(Q) \leq c_1$ then the following hold:
        \begin{enumerate}
            \item $\R^n_Q\cap B(x_Q,\frac{1}{5}\ell(Q))\not=\varnothing$, \text{ and} \label{i:norm-nearby}
            \item $c_Q \asymp 1$. \label{i:cq-close-to-one} 
        \end{enumerate}
\end{lemma}
\begin{proof}
    The proof mostly follows the proof of \cite{To09} Lemma 3.1. We first prove \ref{i:norm-nearby}. Let $\varphi:Z_Q\rightarrow\R$ be Lipschitz satisfying
    \begin{enumerate}
        \item $\chi_{B(x_Q,\frac{1}{100}\ell(Q))} \leq \varphi \leq 100\max\{0,\frac{1}{10} - \frac{\dist(\cdot,x_Q)}{\ell(Q)}\} \leq 100\chi_{B(x_Q,\frac{1}{10}\ell(Q))}$ \text{ and}
        \item $\Lip(\varphi) \lesssim \ell(Q)^{-1}$.
    \end{enumerate}
    Then $\|\varphi\cdot\dist(\cdot,\R^n_Q)\|_\infty \lesssim \ell(Q)$, $\Lip(\varphi\cdot\dist(\cdot,\R^n_Q)) \lesssim 1$, and $\varphi\cdot\dist(\cdot,\R^n_Q) \lesssim 
    \max\{0,5\ell(Q) - \dist(\cdot,x_Q)\}$ so that $c\varphi\cdot\dist(\cdot,\R^n_Q)\in\Lip(5B_Q)$ for some controlled constant $c > 0$. This means 
    \begin{equation*}
        \left| \int \varphi(x)\dist(x,\R^n_Q)\ d\scH^n_{E_Q}(x)\right| \lesssim \alpha(Q)\ell(Q)^{n+1}.
    \end{equation*}
    But, using the fact that $\scH^n(E\cap c_0B_Q) \gtrsim \ell(Q)^n$, we also have
    \begin{align*}
        \int \varphi(x)\dist(x,\R^n_Q)\ d\scH^n_{E_Q}(x) &\geq \dist(\spt(\varphi),\R^n_Q)\int\varphi\ d\scH^n_{E_Q}\\
        &\gtrsim \dist(\spt(\varphi),\R^n_Q)\ell(Q)^n.
    \end{align*}
    By choosing $c_1$ sufficiently small, we get $\dist(\spt(\varphi),\R^n_Q) \leq \frac{1}{10}\ell(Q)$ from which we get $\R^n_Q\cap B(x_Q,\frac{1}{5}\ell(Q)) \not=\varnothing$. 
    
    Now, we prove \ref{i:cq-close-to-one}. Indeed, if $\psi:Z_Q\rightarrow\R$ is Lipschitz with 
    \begin{enumerate}
        \item $\chi_{B(x_Q,\frac{1}{2}\ell(Q))} \leq \psi \leq 10\max\{0,1-\frac{\dist(\cdot,x_Q)}{\ell(Q)}\}\leq 10\chi_{B(x_Q,\ell(Q))}$ \text{ and}
        \item $\Lip(\psi) \lesssim \ell(Q)^{-1}$,
    \end{enumerate}
    then
    \begin{equation*}
        \left| \int\psi\ d\scH^n_{E_Q} - c_Q\int \psi\ d\scH^n_{\|\cdot\|_Q}\right| \lesssim \alpha_{E,X}(Q)\ell(Q)^n.
    \end{equation*}
    This gives
    \begin{equation}\label{e:cQ-large}
        \int\psi\ d\scH^n_{E_Q} - C\alpha_{E,X}(Q)\ell(Q)^n \leq c_Q \int \psi\ d\scH^n_{\|\cdot\|_Q} \leq \int \psi\ d\scH^n_{E_Q} + C\alpha_{E,X}(Q)\ell(Q)^n.
    \end{equation}
    Observe that the fact that $\R^n_Q\cap B(x_Q,\frac{1}{5}\ell(Q))\not=\varnothing$ implies $\int\psi\ d\scH^n_{\R^n_{\|\cdot\|_Q}} \asymp \ell(Q)^n$ because $B(x_Q,\frac{1}{2}\ell(Q))$ contains a ball of radius $\gtrsim \ell(Q)$ contained in $\R^n_{\|\cdot\|_Q}$. Using this and the upper Ahlfors $n$-regularity of $X$ (and hence of $E$), the second inequality of \eqref{e:cQ-large} implies $c_Q \lesssim 1$. The first inequality of \eqref{e:cQ-large} implies
    \begin{equation*}
        c_Q\ell(Q)^n \gtrsim c_Q \int \psi\ d\scH^n_{\|\cdot\|_Q} \geq \scH^n_E\left(\frac{1}{2}B_Q\right) - C\alpha_{E,X}(Q)\ell(Q)^n \gtrsim \ell(Q)^n
    \end{equation*}
    which gives $c_Q \gtrsim 1$ as long as $c_1$ is sufficiently small.
\end{proof}

For the next lemma, we introduce a Gromov-Hausdorff variant of the bilateral $\beta_1$ number of David and Semmes.
\begin{definition}[Gromov-Hausdorff $b\beta_1$]
    Let $x\in X$ and $0 < r < \diam(X)$ and define
    \begin{align*}
        b\beta_1^\GH(x,r) = \frac{1}{r^n}\inf_{\|\cdot\|,Z}\inf_{\substack{ \iota_1:B(x,r)\rightarrow Z\\ \iota_2:(\R^n,\|\cdot\|)\rightarrow Z}}\int_{B(\iota_1(x),r)}&\frac{\dist(y,\iota_2(\R^n))}{r}\ d\iota_1\#[\scH^n_X](y)\\
        &+ \int_{B(\iota_1(x),r)}\frac{\dist(x,\iota_1(X))}{r}\ d\iota_2\#[\scH^n_{\|\cdot\|}](x)
    \end{align*}
    where the infima are taken over all norms $\|\cdot\|$, complete, separable metric spaces $Z$, and isometric embeddings $\iota_1$ and $\iota_2$.
\end{definition}

\begin{lemma}[cf. \cite{To09} Lemma 3.2]\label{l:alpha-bilat}
    Let $X$ be Ahlfors $n$-regular. For any $Q\in\scD(X)$,
    \begin{equation*}
        \alpha_X(Q) \gtrsim b\beta_1^{\GH}(3B_Q) \gtrsim \xi(B_Q)^{n+1}.
    \end{equation*}
\end{lemma}
\begin{proof}
    The proof mostly follows the proof of Lemma 3.2 in \cite{To09}. We may assume that $\alpha(Q) \leq c_1$. Let $\varphi:Z_Q\rightarrow\R$ satisfy $\chi_{B(x_Q,3\ell(Q))} \leq \varphi \leq \max\{0,5-\frac{\dist(\cdot,x_Q)}{\ell(Q)}\}$ with $\Lip(\varphi)\lesssim \ell(Q)^{-1}$. The function
    \begin{equation*}
        f(z) = \left[\dist(z,\R^n_Q) - \dist(z,X_Q)\right]\varphi(z)
    \end{equation*}
    has $\Lip(f) \lesssim 1$ and $\|f\|_\infty \lesssim \ell(Q)$ so that
    \begin{align}
        \alpha(Q)\ell(Q)^{n+1} &\gtrsim \left|\int f\ d\scH^n_{X_Q} - c_Q \int f\ d\scH^n_{\R^n_Q}\right|\nonumber\\
        &= \left|\int \varphi(z)\dist(z,\R^n_Q)\ d\scH^n_{X_Q}(z) + c_Q \int \varphi(z)\dist(z,X_Q)\ d\scH^n_{\R^n_Q}(z)\right|.\nonumber\\
        & \geq \min(1,c_Q)b\beta_{1}^{\GH}(3B_Q)\ell(Q)^{n+1} \gtrsim b\beta_{1}^{\GH}(3B_Q)\ell(Q)^{n+1}.\label{e:bilat-1}
    \end{align}
    where the last line follows because we've assumed $\alpha(Q) \leq c_1$ so that Lemma \ref{l:alpha-props} implies $c_Q \asymp 1$. Now suppose that $\iota_1^Q,\iota_2^Q,\|\cdot\|_Q,Z_Q$ are almost-minimizing for $b\beta_{1}^{\GH}(3B_Q)$ in the sense of \eqref{e:almost-min}. Let $x_0\in 2B_Q\cap X_Q$ be such that $\dist(x_0,\R^n_Q)$ is maximized and let $y_0\in 2B_Q\cap \R^n_Q$ be such that $\dist(y_0,X_Q)$ is maximized. Then
    \begin{equation*}
    d_H^{Z^Q}(2B_Q\cap X_Q, 2B_Q\cap \R^n_Q) \lesssim \dist(x_0,2B_Q\cap\R^n_Q) + \dist(y_0,2B_Q\cap X_Q) =\vcentcolon d_1 + d_2
    \end{equation*}
    so that the right-hand side of \eqref{e:bilat-1} is bounded below by
    \begin{align*}
        \int_{B(x_0,d_1/5)}&\dist(z,\R^n_Q)\ d\scH^n_{X_Q}(z) + \int_{B(y_0,d_2/5)}\dist(z,X_Q)\ d\scH^n_{\|\cdot\|_Q}(z)\\
        &\gtrsim d_1^{n+1} + d_2^{n+1} \gtrsim d_H^{Z_Q}(2B_Q\cap X_Q, 2B_Q\cap \R^n_Q)^{n+1}.
    \end{align*}
    so that $b\beta_{1}^\GH(3B_Q)^{1/(n+1)}\ell(Q) \gtrsim d_H^{Z_Q}(2B_Q\cap X_Q, 2B_Q\cap \R^n_Q)$. But this means that there exists a $\lesssim b\beta_{1}^\GH(3B_Q)^{1/(n+1)}\ell(Q)$-isometry between $B_Q\cap X_Q$ and $B_{\|\cdot\|_Q}(0,\ell(Q))$ implying
    \begin{equation*}
        \xi_{X}(B_Q) \lesssim \frac{1}{\ell(Q)}b\beta_{1}^\GH(3B_Q)^{1/(n+1)}\ell(Q) = b\beta_{1}^\GH(3B_Q)^{1/(n+1)}.\qedhere
    \end{equation*}
\end{proof}

\subsubsection{Alpha numbers in uniformly rectifiable metric spaces}
Unless stated otherwise, for the rest of this section fix a complete metric space $(Z,d)$, an Ahlfors $n$-regular metric space $X\subseteq Z$, and an $L$-bi-Lipschitz mapping $g:\R^n\rightarrow Z$. Set $\Sigma = g(\R^n)$ and let $E\subseteq X\cap\Sigma$ be $\scH^n$-measurable.

We wish to control $\alpha_{E,X}$ in terms of controlled quantities associated to the bi-Lipschitz image $\Sigma$. Given any cube $I\subseteq\R^n$, let $\|\cdot\|_I$ be a norm such that $\sup_{x,y\in I} |\dist(g(x),g(y)) - \|x-y\|_I|\leq 2\md_g(I)\ell(I)$. Let 
\begin{equation*}
    \td{Z} = (X\cup\Sigma)\sqcup\R^n
\end{equation*}
and define a metric $\zeta_I:\td{Z}\times \td{Z}\rightarrow \R$ by
\begin{equation}\label{e:zeta-def-2}
        \zeta_I(x,y) = \begin{cases}
            d(x,y) & \text{ if $x,y\in X\cup \Sigma$},\\
            \|x - y\|_I & \text{ if $x,y\in \R^n$},\\
            \inf_{u\in I}\|x-u\|_I + 2\md_g(I)\ell(I) + d(g(u),y) & \text{ if $x\in \R^n$ and $y\in X\cup\Sigma$}.
        \end{cases}
    \end{equation}
For proof that $\zeta_I$ is a metric, the reader can see \cite{Ba23} Lemma 2.24 for a nearly identical case. The basic idea is that since the mapping $g$ makes an error of at most $2\md_g(I_Q)\ell(Q)$, we must add this quantity when passing between $\R^n$ and $X\cup\Sigma\subseteq Z$ to prevent shortcuts which break the triangle inequality. The metric space $(\td{Z},\zeta_I)$ will serve as a test space for controlling $\alpha_{E,X}$ near scales and locations in $E$ nearby $g(I)$. We begin by controlling $\alpha_{E,X}$ by a distance of measures more directly adapted to the dyadic lattices in the domain of $g$.

For any cube $I\subseteq\R^n$, we let $h_I = \frac{d\scL^n}{d\scH^n_{\|\cdot\|_I}}$ be the Radon-Nikodym derivative of $\scL^n$ with respect to $\scH^n_{\|\cdot\|_I}$. We define
\begin{equation*}
    P \vcentcolon= g^{-1}(E),
\end{equation*}
and we set
\begin{equation*}
    c_{P,I} \vcentcolon= h_I\fint_{I}\scJ_g\chi_P(u) d\scL(u).
\end{equation*}
With this setup, we can define cube-adapted alpha numbers.
\begin{definition}
    Let $\mu$ and $\nu$ be Borel measures of bounded support in $Z$. Set $D = \max\{\diam(\spt(\mu)),\diam(\spt(\nu))\}$ and define
    \begin{equation*}
        \widetilde{\dist}_Z(\mu,\nu) \vcentcolon= \sup\left\{\int fd\mu - \int fd\nu : f:Z\rightarrow[-D,D],\ \Lip(f)\leq 1 \right\}.
    \end{equation*}
\end{definition}
\begin{definition}
    For any cube $I\subseteq\R^n$, set
    \begin{equation*}
        \td{\alpha}_E(I) \vcentcolon= \frac{1}{\ell(I)^{n+1}}\widetilde{\dist}_{(\td{Z},\zeta_I)}(\scH^n|_{g(I)\cap E},c_{I,P}\scH^n|_I).
    \end{equation*}
    where we identify $g(I)\cap E$ and $I$ with their isometric embeddings inside $\td{Z}$
\end{definition}
Now, fix $R\in\scD(X)$ and let $Q\in\scD(R)$. Let $\td{\cD} = \td{\cD}(I_R)$ so that whenever $Q\cap E\not=\varnothing$ there always exists some $I_Q\in\td{\cD}$ such that $g^{-1}(10B_Q)\subseteq I_Q$ and $\ell(I_Q)\lesssim_L \ell(Q)$.
\begin{lemma}\label{l:dyadic-alpha}
    For any $Q\in\scD(R)$,
    \begin{equation*}
        \alpha_{E,X}(Q) \lesssim_{L,n} \td{\alpha}_E(I_Q)
    \end{equation*}
\end{lemma}
\begin{proof}
    For convenience, we write $\|\cdot\|_Q \vcentcolon= \|\cdot\|_{I_Q}, c_{P,Q}\vcentcolon= c_{P,I_Q}$, and so on. Because $X$ and $(\R^n,\|\cdot\|_Q)$ embed isometrically in the complete metric space ($\td{Z}_Q,\zeta_Q$) and $\ell(I_Q)\asymp_{L}\ell(Q)$, we only need to show
    \begin{equation*}
        \dist_{5B_Q}(\scH^n_E,c_{P,Q}\scH^n_{\|\cdot\|_Q}) \lesssim_{L,n} \widetilde{\dist}_{\td{Z}_Q}(\scH^n|_{g(I_Q)\cap E}, c_{P,Q}\scH^n|_{I_Q})
    \end{equation*}
    where all sets, balls, and measures are taken in $(\td{Z}_Q,\zeta_Q)$. Any $f\in\Lip_1(5B_Q)$ is also an admissible Lipschitz function for the supremum in the definition of $\widetilde{\dist}$. Therefore,  using the facts that $f = 0$ outside of $5B_Q$, $5B_Q\cap E\subseteq g(I_Q)\cap E$, and $5B_Q\cap \R^n \subseteq I_Q$, for any $f\in\Lip_1(5B_Q)$ we have
    \begin{align*}
        \int_{5B_Q} fd\scH^n_E - c_{P,Q}\int_{5B_Q} fd\scH^n_{\|\cdot\|_Q} &= \int f d\scH^n|_{g(I_Q)\cap E} - c_{P,Q}\int fd\scH^n|_{I_Q}\\
        &\leq \widetilde{\dist}_{\td{Z}_Q}(\scH^n|_{g(I_Q)\cap E}, c_{P,Q}\scH^n|_{I_Q}).\qedhere
    \end{align*}
\end{proof}
We can now record our primary estimate for $\alpha_{E,X}$. It is a simple consequence of the following lemma that $\alpha$ has a weak Carleson condition on bi-Lipschitz images.
\begin{lemma}[cf. \cite{To09} (4.1) ]\label{l:bi-lip-alpha}
    For any $Q\in\scD(R)$, we have
    \begin{equation}\label{e:bi-lip-alph}
        \alpha_{E,X}(Q) \lesssim_{L,n} \md_g(I_Q) + \sum_{I\subseteq I_Q}\frac{\ell(I)^{1+n/2}}{\ell(I_Q)^{1+n}}\|\Delta_I(\scJ_g\chi_{g^{-1}(E)})\|_2.
    \end{equation}
\end{lemma}
\begin{proof}
    The proof follows the strategy of that of Theorem 1.1 in \cite{To09}. By Lemma \ref{l:dyadic-alpha}, we only need to show the desired bound for $\td{\alpha}_E(I_Q)$ in place of $\alpha_{E,X}(Q)$.
    
    Recall $P = g^{-1}(E)$. Using the definition of $c_{P,Q}$ and the area formula, we have
    \begin{align*}
        \bigg|\int_{g(I_Q)\cap E}f(z)&d\scH^n(z) - c_{P,Q}\int_{I_Q}f(u)d\scH^n_{\|\cdot\|_Q}(u)\bigg| \\
        &= \left|\int_{I_Q\cap P} f(g(u))\scJ_g(u)d\scL(u) - \fint_{I_Q}(\scJ_g\chi_P)d\scL \int_{I_Q}f(u)h_Qd\scH^n_{\|\cdot\|_Q}(u)\right|\\
        &\leq \left|\int_{I_Q} [f(g(u)) - f(u)]\scJ_g\chi_P(u)d\scL(u)\right| + \left|\int_{I_Q} \left[\scJ_g\chi_P(u) - \fint_{I_Q}\scJ_g\chi_P \right]f(u)d\scL(u)\right|\\
        &\leq C(L,n)\md_g(I_Q)\ell(Q)^{n+1} + \sum_{I\subseteq I_Q}\left|\int_{I_Q} f(u)\Delta_I(\scJ_g\chi_P)(u)d\scL(u)\right|
    \end{align*}
    where the final inequality uses that $g$ is a $C(L)\md_g(I_Q)\ell(Q)$-isometry, $\|f\|_\infty + \|\scJ_g\|_\infty\lesssim_{L,n} 1$, and $\scH^n(I_Q) \lesssim_{L,n} \ell(Q)^n$. For any cube $I$, let $u_I$ denote the center of $I$. Since $\Delta_I(\scJ_g\chi_P)$ has mean zero, we have
    \begin{align*}
        \sum_{I\subseteq I_Q}\left|\int_{I_Q} f(u)\Delta_I(\scJ_g\chi_P)(u)d\scL(u)\right| &= \sum_{I\subseteq I_Q} \left| \int_{I_Q} [f(u) - f(u_I)]\Delta_I(\scJ_g\chi_P)(u)d\scL(u)\right|\\
        &\lesssim \sum_{I\subseteq I_Q}\ell(I)\|\Delta_I(\scJ_g\chi_P)\|_1 \leq \sum_{I\subseteq I_Q}\ell(I)^{1+n/2}\|\Delta_I(\scJ_g\chi_P)\|_2.
    \end{align*}
    where the final line follows from Cauchy-Schwarz.
\end{proof}
\begin{remark}\label{rem:wavlet-sgl}
    The second term in \eqref{e:bi-lip-alph} satisfies the following strong Carleson condition: For any $J\in\cD(\R^n)$
    \begin{equation*}
        \sum_{I\subseteq J}\left(\sum_{I'\subseteq I}\frac{\ell(I')^{1+n/2}}{\ell(I)^{1+n}}\|\Delta_{I'}(\scJ_g\chi_P)\|_2\right)^2\ell(I)^n \lesssim_{L,n} \ell(J)^n.
    \end{equation*}
    The primary tool for the simple proof is Cauchy-Schwarz (see the proof of Theorem 1.1 in \cite{To09}).
\end{remark}
\begin{proposition}\label{p:bilip-alpha-wcc}
     $\alpha_{E,X}\in\WC(X)$.
\end{proposition}
\begin{proof}
    Fix $\epsilon > 0$ and set
    \begin{align*}
        \scG &= \left\{Q\in\scD(R) : \md_g(I_Q) < \epsilon,\ \sum_{I\subseteq I_Q}\frac{\ell(I)^{1+n/2}}{\ell(I_Q)^{1+n}}\|\Delta_{I_Q}(\scJ_g\chi_P)\|_2 < \epsilon\right\}\\ƒ
        \scB &= \scD(R) \setminus\scG.
    \end{align*}
    By Lemma \ref{l:bi-lip-alpha}, for any $Q\in\scG$ we have $\alpha_{E,X}(Q) \lesssim_{L,n} \epsilon$. Therefore, it suffices to show that $\scB$ is Carleson. This follows from Theorem \ref{t:qdiff}, Remark \ref{rem:wavlet-sgl}, and the fact that $\td{\cD}$ consists of finitely many shifted lattices.
\end{proof}

\begin{corollary}\label{c:bilip-alpha-wcc}
    $\alpha_{E,\Sigma}\in \WC(\Sigma)$.
\end{corollary}
\begin{proof}
    Take $X = Z = \Sigma$ in Proposition \ref{p:bilip-alpha-wcc}
\end{proof}

\begin{proposition}\label{p:alpha-carleson}
    Let $X$ be a uniformly $n$-rectifiable metric space. Then $\alpha_X\in \WC(X)$. 
\end{proposition}
\begin{proof}
    We will use Theorem \ref{t:wgl-suff-cond}. Without loss of generality, assume $X\subseteq\ell_\infty$ and recall that Theorem \ref{t:vbpbi} implies that $X\in\VBP(\scS)$ inside $\ell_\infty$, so Theorem \ref{t:wgl-suff-cond} implies that it suffices to show that for any $\Sigma\in\scS$ there exists $\beta_\Sigma\in\WC(X)$ satisfying \eqref{e:transfer}. We claim that $\alpha_{X\cap \Sigma,X}$ is such a coefficient. Indeed, $\alpha_{X\cap\Sigma,X}\in\Co(\ell_\infty)$ by Lemma \ref{l:alph-monotone} (with Remark \ref{rem:alpha-co-z}) and Lemma \ref{l:transfer} implies that for any $Q\in\scD(X)$ 
    \begin{equation*}
        \alpha_X(c_0B_Q) \lesssim \frac{\scH^n(Q\setminus \Sigma)}{\ell(Q)^n} + \alpha_{X\cap\Sigma, X}(c_0B_Q).
    \end{equation*}
    Since $\alpha_{X\cap\Sigma, X} \in \WC(X)$ by Proposition \ref{p:bilip-alpha-wcc}, the result follows.
\end{proof}
It now follows easily that a weak Carleson condition for $\alpha$ is equivalent to uniform rectifiability for Ahlfors regular metric spaces.
\begin{proof}[Proof of Theorem \ref{t:metric-UR-alpha}]
    The forward direction is exactly the statement of Proposition \ref{p:alpha-carleson}. For the backward direction, notice that Lemma \ref{l:alpha-bilat} immediately gives a weak Carleson condition for $\xi$. This is exactly the BWGL, which is equivalent to uniform rectifiability by Theorem \ref{t:BHS-BWGL}.
\end{proof}

\bibliographystyle{alpha}
\bibliography{bib-file-RST}

\begin{thebibliography}{CGLT16}

\bibitem[AH22]{AH22}
Jonas Azzam and Matthew Hyde.
\newblock The weak lower density condition and uniform rectifiability.
\newblock {\em Ann. Fenn. Math.}, 47(2):791--819, 2022.

\bibitem[AS14]{AS14}
Jonas Azzam and Raanan Schul.
\newblock A quantitative metric differentiation theorem.
\newblock {\em Proc. Amer. Math. Soc.}, 142(4):1351--1357, 2014.

\bibitem[ATT20]{ATT20}
Jonas Azzam, Xavier Tolsa, and Tatiana Toro.
\newblock Characterization of rectifiable measures in terms of {$\alpha$}-numbers.
\newblock {\em Trans. Amer. Math. Soc.}, 373(11):7991--8037, 2020.

\bibitem[Bat23]{Ba23}
David Bate.
\newblock On 1-regular and 1-uniform metric measure spaces.
\newblock Youtube, 2023.
\newblock https://www.youtube.com/watch?v=9zXjfKeWGv4\&t=2096s.

\bibitem[Bes28]{Be28}
A.S. Besicovitch.
\newblock On the fundamental geometrical properties of linearly measurable plane sets of points.
\newblock {\em Mathematische Annalen}, 98:422--464, 1928.

\bibitem[Bes38]{Be38}
A.S. Besicovitch.
\newblock On the fundamental geometrical properties of linearly measurable plane sets of points (ii).
\newblock {\em Mathematische Annalen}, 115:296--329, 1938.

\bibitem[BHS23]{BHS23}
David Bate, Matthew Hyde, and Raanan Schul.
\newblock Uniformly rectifiable metric spaces: Lipschitz images, bi-lateral weak geometric lemma and corona decompositions.
\newblock arXiv preprint, arXiv:2306.12933, 2023.

\bibitem[CGLT16]{CGLT16}
Vasileios Chousionis, John Garnett, Triet Le, and Xavier Tolsa.
\newblock Square functions and uniform rectifiability.
\newblock {\em Trans. Amer. Math. Soc.}, 368(9):6063--6102, 2016.

\bibitem[Chr90]{Ch90}
M.~Christ.
\newblock A {T}(b) theorem with remarks on analytic capacity and the {C}auchy integral.
\newblock {\em Colloq. Math.}, 60/61(2):601--628, 1990.

\bibitem[CMT20]{CMT20}
Vasilis Chousionis, Valentino Magnani, and Jeremy~T. Tyson.
\newblock On uniform measures in the {H}eisenberg group.
\newblock {\em Adv. Math.}, 363:106980, 42, 2020.

\bibitem[Dab20]{Dam20}
Damian Dabrowski.
\newblock Necessary condition for rectifiability involving {W}asserstein distance {$W_2$}.
\newblock {\em Int. Math. Res. Not. IMRN}, (22):8936--8972, 2020.

\bibitem[Dab21]{Dam21}
Damian Dabrowski.
\newblock Sufficient condition for rectifiability involving {W}asserstein distance {$W_2$}.
\newblock {\em J. Geom. Anal.}, 31(8):8539--8606, 2021.

\bibitem[Dav88]{Da88}
Guy David.
\newblock Morceaux de graphes lipschitziens et int\'{e}grales singuli\`eres sur une surface.
\newblock {\em Rev. Mat. Iberoamericana}, 4(1):73--114, 1988.

\bibitem[DL08]{DeLe08}
Camillo De~Lellis.
\newblock {\em Rectifiable sets, densities and tangent measures}.
\newblock Zurich Lectures in Advanced Mathematics. European Mathematical Society (EMS), Z\"{u}rich, 2008.

\bibitem[DS91]{DS91}
Guy David and Stephen Semmes.
\newblock Singular integrals and rectifiable sets in {${\bf R}^n$}: {B}eyond {L}ipschitz graphs.
\newblock {\em Ast\'{e}risque}, (193):152, 1991.

\bibitem[DS93]{DS93}
Guy David and Stephen Semmes.
\newblock {\em Analysis of and on uniformly rectifiable sets}, volume~38 of {\em Mathematical Surveys and Monographs}.
\newblock American Mathematical Society, Providence, RI, 1993.

\bibitem[FV23]{FV23}
Katrin Fassler and Ivan~Yuri Violo.
\newblock On various carleson-type geometric lemmas and uniform rectifiability in metric spaces.
\newblock arXiv preprint, arXiv:2310.10519, 2023.

\bibitem[Hah05]{Ha05}
Immo Hahlomaa.
\newblock Menger curvature and {L}ipschitz parametrizations in metric spaces.
\newblock {\em Fund. Math.}, 185(2):143--169, 2005.

\bibitem[HM12]{HM12}
T.~Hyt{\"o}nen and H.~Martikainen.
\newblock Non-homogeneous ${T}b$ theorem and random dyadic cubes on metric measure spaces.
\newblock {\em J. Geom. Anal.}, 22(4):1071--1107, 2012.

\bibitem[Kir94]{Ki94}
Bernd Kirchheim.
\newblock Rectifiable metric spaces: local structure and regularity of the {H}ausdorff measure.
\newblock {\em Proc. Amer. Math. Soc.}, 121(1):113--123, 1994.

\bibitem[KP87]{KP87}
Old\v{r}ich Kowalski and David Preiss.
\newblock Besicovitch-type properties of measures and submanifolds.
\newblock {\em J. Reine Angew. Math.}, 379:115--151, 1987.

\bibitem[KP02]{KP02}
Bernd Kirchheim and David Preiss.
\newblock Uniformly distributed measures in {E}uclidean spaces.
\newblock {\em Math. Scand.}, 90(1):152--160, 2002.

\bibitem[Ler03]{Le03}
Gilad Lerman.
\newblock Quantifying curvelike structures of measures by using $l_2$ {J}ones quantities.
\newblock {\em Comm. Pure Appl. Math.}, 56(9):1294--1365, 2003.

\bibitem[Lor03]{Lo03}
Andrew Lorent.
\newblock Rectifiability of measures with locally uniform cube density.
\newblock {\em Proc. London Math. Soc. (3)}, 86(1):153--249, 2003.

\bibitem[Mar61]{Mar61}
J.~M. Marstrand.
\newblock Hausdorff two-dimensional measure in {$3$}-space.
\newblock {\em Proc. London Math. Soc. (3)}, 11:91--108, 1961.

\bibitem[Mat75]{Ma75}
Pertti Mattila.
\newblock Hausdorff {$m$} regular and rectifiable sets in {$n$}-space.
\newblock {\em Trans. Amer. Math. Soc.}, 205:263--274, 1975.

\bibitem[Mer22]{Me22}
Andrea Merlo.
\newblock Geometry of 1-codimensional measures in {H}eisenberg groups.
\newblock {\em Invent. Math.}, 227(1):27--148, 2022.

\bibitem[Nim17]{Ni17}
A.~Dali Nimer.
\newblock A sharp bound on the {H}ausdorff dimension of the singular set of a uniform measure.
\newblock {\em Calc. Var. Partial Differential Equations}, 56(4):Paper No. 111, 31, 2017.

\bibitem[Nim19]{Ni19}
A.~Dali Nimer.
\newblock Uniformly distributed measures have big pieces of {L}ipschitz graphs locally.
\newblock {\em Ann. Acad. Sci. Fenn. Math.}, 44(1):389--405, 2019.

\bibitem[Nim22]{Ni22}
A.~Dali Nimer.
\newblock Conical 3-uniform measures: a family of new examples and characterizations.
\newblock {\em J. Differential Geom.}, 121(1):57--99, 2022.

\bibitem[Oki92]{Ok92}
Kate Okikiolu.
\newblock Characterization of subsets of rectifiable curves in $\mathbf{R}^n$.
\newblock {\em Journal of the London Mathematical Society}, s2-46(2):336--348, 1992.

\bibitem[Pre87]{Pr87}
David Preiss.
\newblock Geometry of measures in {${\bf R}^n$}: distribution, rectifiability, and densities.
\newblock {\em Ann. of Math. (2)}, 125(3):537--643, 1987.

\bibitem[PT92]{PT92}
David Preiss and Jaroslav Ti\v{s}er.
\newblock On {B}esicovitch's {$\frac12$}-problem.
\newblock {\em J. London Math. Soc. (2)}, 45(2):279--287, 1992.

\bibitem[Sch07]{Sc07}
Raanan Schul.
\newblock Ahlfors-regular curves in metric spaces.
\newblock {\em Ann. Acad. Sci. Fenn. Math.}, 32(2):437--460, 2007.

\bibitem[Sch09]{Sc09}
Raanan Schul.
\newblock Bi-{L}ipschitz decomposition of {L}ipschitz functions into a metric space.
\newblock {\em Rev. Mat. Iberoam.}, 25(2):521--531, 2009.

\bibitem[Tol09]{To09}
Xavier Tolsa.
\newblock Uniform rectifiability, {C}alder\'on-{Z}ygmund operators with odd kernel, and quasiorthogonality.
\newblock {\em Proc. Lond. Math. Soc. (3)}, 98(2):393--426, 2009.

\bibitem[Tol12]{To12}
Xavier Tolsa.
\newblock Mass transport and uniform rectifiability.
\newblock {\em Geom. Funct. Anal.}, 22(2):478--527, 2012.

\bibitem[Tol15]{To15}
Xavier Tolsa.
\newblock Uniform measures and uniform rectifiability.
\newblock {\em J. Lond. Math. Soc. (2)}, 92(1):1--18, 2015.

\bibitem[TT15]{TT15}
Xavier Tolsa and Tatiana Toro.
\newblock Rectifiability via a square function and {P}reiss' theorem.
\newblock {\em Int. Math. Res. Not. IMRN}, (13):4638--4662, 2015.

\end{thebibliography}
\end{document}